\documentclass[12pt, reqno]{amsart}
\usepackage{amsmath, amsthm, amscd, amsfonts, amssymb, graphicx, xcolor}
\usepackage[bookmarksnumbered, colorlinks, plainpages]{hyperref}
\usepackage{tikz}
\usetikzlibrary{arrows.meta, positioning}
\usepackage{tikz-cd}
\usepackage{enumitem}
\usepackage{xcolor}
\definecolor{RED}{named}{red}

\newtheorem{theorem}{Theorem}[section]
\newtheorem{lemma}[theorem]{Lemma}
\newtheorem{proposition}[theorem]{Proposition}
\newtheorem{corollary}[theorem]{Corollary}
\theoremstyle{definition}
\newtheorem{definition}[theorem]{Definition}
\newtheorem{example}[theorem]{Example}

\theoremstyle{remark}
\newtheorem{remark}[theorem]{Remark}
\numberwithin{equation}{section}

\theoremstyle{theorem}
\newtheorem*{theorem*}{Theorem}
\newtheorem*{proposition*}{Proposition}
\newtheorem*{corollary*}{Corollary}
\newtheorem*{lemma*}{Lemma}
\newtheorem*{example*}{Example}
\theoremstyle{remark}
\newtheorem*{remark*}{Remark}

\newcommand{\grouphull}[1]{\left\langle #1 \right\rangle}

\def\subfin{{\operatorname{Sub_{fin}}}}

\def\Eig{{\rm Eig}}
\newcommand{\envelopingOdometer}[1]{\mathcal{O}(#1)}

\def\finite{{\mathcal{F}}} 
\def\hyper{{\mathcal{H}}}

\usepackage{stmaryrd}
\newcommand{\eigenhull}[1]{\left\llbracket #1 \right\rrbracket}

\begin{document}
\color{black}

\title
[Induced dynamics and quasifactors for subodometers]
{Induced dynamics and quasifactors for minimal equicontinuous actions on Stone spaces}

\author[]{Mar\'{\i}a Isabel Cortez}
\address{Facultad de Matem\'aticas, Pontificia Universidad Cat\'olica de Chile. Edificio Rolando Chuaqui, Campus San Joaquín. Avda. Vicuña Mackenna 4860, Macul, Chile.}
\email{maria.cortez@uc.cl}

\author[]{Till Hauser}
\address{Facultad de Matem\'aticas, Pontificia Universidad Cat\'olica de Chile. Edificio Rolando Chuaqui, Campus San Joaquín. Avda. Vicuña Mackenna 4860, Macul, Chile.}
\email{hauser.math@mail.de}

\thanks{This article was funded by the Deutsche Forschungsgemeinschaft (DFG, German Research Foundation) – 530703788.}

\begin{abstract}
    A minimal equicontinuous action of a group $G$ on a Stone space $X$ is called a \emph{subodometer}. 
    If such a subodometer arises from a group rotation, we refer to it as an \emph{odometer}.
    For subodometers $(X,G)$ we show that the hyperspace $\hyper(X)$ - given by all closed subsets of $X$ and the Vietoris topology - decomposes into subodometers.     
    We show that an infinite subodometer is an odometer if and only if $\hyper(X)$ decomposes into factors of $(X,G)$. 
    
    Similarly, we consider $\mathcal{M}(X)$, the space of regular Borel probability measures equipped with the weak-* topology. 
    We show that for a subodometer $(X,G)$ also the connected space $\mathcal{M}(X)$ decomposes into subodometers.
    We prove that an infinite subodometer $(X,G)$ is an odometer if and only if $\mathcal{M}(X)$ decomposes into factors of $(X,G)$. 
    For this, we study different notions of regular recurrence. 
    
    Furthermore, we study the disjointness of minimal actions to subodometers and show that this disjointness can be detected from the pairwise disjointness of finite factors.     
    Using this we prove that a minimal action is disjoint from all subodometers if and only if it has a connected maximal equicontinuous factor. 
\end{abstract} \maketitle





\section{Introduction}
The notion of a quasifactor was introduced by Glasner in \cite{glasner1975compressibility} as a tool to study disjointness. In that work, he showed that the quasifactors of a minimal distal system $(X,G)$ are factors of its Ellis semigroup, and that a minimal system $(Y,G)$ is disjoint from $(X,G)$ whenever $(Y,G)$ has no common factors with the Ellis semigroup of $(X,G)$. Subsequent works have further developed these ideas by establishing connections between the dynamical properties of a system and those of the induced flows on its hyperspace and on its space of probability measures. For instance, several results describe the nature of quasifactors in relation to the dynamical properties of the underlying system \cite{auslander2000ellis,glasner2000quasifactors,nagar2022characterization,vanderWoude1979disjointness}, as well as connections with invariants such as entropy \cite{glasner1995quasifactors, kwietniak2007topological}, where we only mention some of the results. 
In this work, we focus on minimal equicontinuous actions on Stone spaces.
We distinguish those arising from group rotations from those that do not through the nature of their quasifactors and the minimal components of the action induced on their space of probability measures, leading to several results on disjointness. This work relies on the characterization in terms of eigensets and generating scales provided in \cite{cortez2026minimal}. We use the following terminology.

\begin{definition}
    A minimal equicontinuous action on a Stone space is called \emph{subodometer}. 
    A subodometer is called \emph{odometer} if it is a rotation. 
\end{definition}

Let $(X,G)$ be a subodometer. 
Since $X$ is equicontinuous, $(\hyper(X),G)$ is also equicontinuous and hence decomposes into minimal components. 
Furthermore, since $X$ is a Stone space, $\hyper(X)$ is also a Stone space \cite[Proposition 8.6]{illanes1999hyperspaces}. Thus, the minimal components of $\hyper(X)$ are minimal equicontinuous actions on Stone spaces and hence subodometers.
Note that $\mathcal{M}(X)$ is convex and hence connected. 
Nevertheless, in Section \ref{sec:RegularRecurrence} we explore different concepts of regular recurrence and prove the following. 

\begin{theorem}
\label{the:INTROInducedDynamicsOfSubodometerDecompose}
    For any subodometer $(X,G)$ the actions $(\hyper(X),G)$ and $(\mathcal{M}(X),G)$ decompose into a disjoint union of subodometers. 
\end{theorem}

A minimal subaction of $(\hyper(X),G)$ is called a ($\hyper$-)\emph{quasifactor} \cite{glasner1975compressibility,auslander1988minimal} and it is natural to ask which subodometers appear as quasifactors of a subodometer. 
It is well known\footnote{
    See, for instance, the discussion in the introduction of \cite{glasner1995quasifactors} and \cite[Theorem~7.3]{auslander1988minimal}.} 
that for a minimal distal action any factor appears (up to conjugacy) as a quasifactor. 
Concerning the converse, a first insight is provided by \cite[Corollary 11.20]{auslander1988minimal}, which yields that any quasifactor of an odometer is a factor. 
In Theorem \ref{the:hyperspaceCharacterizationOdometerViaQuasifactors} we show the following converse. 

\begin{theorem*}
    An infinite subodometer $(X,G)$ is an odometer if and only if $(\hyper(X),G)$ decomposes into factors of $(X,G)$. 
\end{theorem*}

Note that a similar statement does not hold for finite subodometers, as we will discuss at the end of Section \ref{sec:hyperspacesOfSubodometers}. 
In Section \ref{sec:inducedDynamicsOnTheRegularBorelProbabilityMeasures} we then study the minimal subaction of $(\mathcal{M}(X),G)$, which we call $\mathcal{M}$-\emph{quasifactors}. 
We show the following in the Theorems \ref{the:factorsAreMQuasifactors}, \ref{the:forOdometersComponentsMXfactor} and \ref{the:MXcharacterizationOdometerViaMinimalComponents}. 

\begin{theorem*}
    Let $(X,G)$ be a subodometer. 
    \begin{itemize}
        \item[(i)] 
        Any factor of $(X,G)$ is conjugated to a $\mathcal{M}$-quasifactor of $(X,G)$.
        \item[(ii)]
        Whenever $(X,G)$ is an odometer, then the $\mathcal{M}$-quasifactors of $(X,G)$ are exactly the factors of $(X,G)$.
        \item[(iii)] 
        An infinite subodometer $(X,G)$ is an odometer if and only if $(\mathcal{M}(X),G)$ decomposes into factors of $(X,G)$. 
    \end{itemize}
\end{theorem*}

Recall that minimal actions $(X,G)$ and $(Y,G)$ are called \emph{disjoint} if the product action $(X\times Y,G)$ is minimal. We write $X\perp Y$ to indicate the disjointness of $(X,G)$ and $(Y,G)$. 
We denote $\Eig(X,G)$ for the set of all finite index subgroups $\Gamma\leq G$ for which there exists a factor map $X\to G/\Gamma$. 
In Section \ref{sec:DisjointnessOfSubodometers} we study the disjointness of minimal actions to subodometers and show the following in Theorem \ref{the:disjointnessCharacterization}. 

\begin{theorem*}[\ref{the:disjointnessCharacterization}]
    Let $(X,G)$ be a minimal action and $(Y,G)$ be a subodometer. 
    The following statements are equivalent. 
    \begin{itemize}
        \item[(i)] 
        $X\perp Y$.
        \item[(ii)] 
        $Z\perp Y$, for every finite factor $Z$ of $X$. 
        \item[(iii)]
        $X \perp Z'$, for every finite factor $Z'$ of $Y$.
        \item[(iv)]
        $Z \perp Z'$, for all finite factors $Z$ and $Z'$ of $X$ and $Y$, respectively.
        \item[(v)]
        $\Lambda\Gamma=G$, for all $\Lambda\in \Eig(X,G)$ and $\Gamma\in \Eig(Y,G)$.
    \end{itemize}
\end{theorem*}

It is well known that the disjointness of two actions implies the absence of nontrivial common factors, while the converse fails in general \cite[Chapter~11]{auslander1988minimal}. 
In Theorem \ref{the:characterizationNNCFsubodometers} we show the following. 

\begin{theorem*}[\ref{the:characterizationNNCFsubodometers}]
    Let $(X,G)$ be a minimal action and $(Y,G)$ be a subodometer. 
    The following statements are equivalent. 
    \begin{itemize}
        \item[(i)] $X$ and $Y$ have no nontrivial common factor.
        \item[(ii)] $X$ and $Y$ have no nontrivial common finite factor.
        \item[(iii)] $\grouphull{\Lambda\cup\Gamma}=G$, for all $\Lambda\in \Eig(X,G)$ and $\Gamma\in \Eig(Y,G)$.
    \end{itemize}    
\end{theorem*}

In \cite[Chapter 11]{auslander1988minimal} it was shown that a minimal action (that admits an invariant regular Borel probability measure) is weakly mixing\footnote{
    An action $(X,G)$ is called \emph{weakly mixing} if every invariant nonempty open subset $U\subseteq X^2$ is dense in $X^2$ \cite{auslander1988minimal}. 
}
if and only if it is disjoint from any equicontinuous action. 
In Theorem \ref{the:characterizationDisjointnessToSubodometers} we classify the minimal actions that are disjoint from all subodometers. 

\begin{theorem*}[\ref{the:characterizationDisjointnessToSubodometers}]
    For a minimal action the following statements are equivalent. 
    \begin{itemize}
        \item[(i)] $X\perp Y$ for all subodometers $(Y,G)$. 
        \item[(ii)] $X\perp Y$ for all odometers $(Y,G)$.
        \item[(iii)] $X\perp Y$ for all finite minimal actions $(Y,G)$. 
        \item[(iv)] $X\perp Y$ for all finite odometers $(Y,G)$.
        \item[(v)] $(X,G)$ has no nontrivial finite factors. 
        \item[(vi)] The maximal equicontinuous factor of $(X,G)$ is connected. 
    \end{itemize}
\end{theorem*}

\subsection*{Convention}
    Throughout the text, $G$ denotes a (discrete) group unless otherwise specified.

\section{Preliminaries}
\label{sec:prelims}

If $R$ is a relation on a set $X$ and $x\in X$ we denote \[R[x]:=\{x'\in X;\, (x',x)\in R\} 
~~\text{ and }~~ R^{-1}:=\{(x',x);\, (x,x')\in R\}.\] 
For relations $R$ and $R'$ on $X$ we denote $RR'$ for the set of all $(x,x')\in X^2$ for which there exists $y\in X$ with $(x,y)\in R$ and $(y,x')\in R'$. Note that this operation is associative. 

\subsection{Topology}
\label{subsec:prelims_topology}
A subset of a topological space is called \emph{clopen} if it is closed and open. A relation on a topological space $X$ is called \emph{closed/open/clopen} if it is closed/open/clopen as a subset of $X^2$ equipped with the product topology. 

\subsubsection{Uniformities of compact Hausdorff spaces}
\label{subsubsec:prelims_topology_UniformitiesOfCompactHausdorffSpaces}
Let $X$ be a compact Hausdorff space. 
We denote $\mathbb{U}_X$ for the \emph{uniformity of $X$}, i.e.\ the set of all neighbourhoods $\epsilon$ of the diagonal $\Delta_X:=\{(x,x);\, x\in X\}$ within $X^2$. 
The elements of $\mathbb{U}_X$ are called \emph{entourages}. 
An entourage $\epsilon\in \mathbb{U}_X$ is called symmetric if $\epsilon=\epsilon^{-1}$.
For all $\epsilon\in \mathbb{U}_X$ there exists a symmetric $\delta\in \mathbb{U}_X$ such that $\delta\delta\subseteq \epsilon$. 
See \cite[Chapter 6]{kelley2017general} for details on uniformities. 

For $\epsilon\in \mathbb{U}_X$ we denote $B_\epsilon(x):=\epsilon[x]$. 
A subset $\mathbb{B}\subseteq \mathbb{U}_X$ is called a \emph{base for $\mathbb{U}_X$} if for each $\epsilon\in \mathbb{U}_X$ there exists $\delta\in \mathbb{B}$ with $\delta\subseteq \epsilon$. 
Whenever $\mathbb{B}$ is a base for $\mathbb{U}_X$, then for $x\in X$ the set $\{B_\epsilon(x);\, \epsilon\in \mathbb{B}\}$ is a (topological) neighbourhood base of $x$ \cite[Corollary 30]{kelley2017general}. This allows one to recover the topology of $X$ from $\mathbb{U}_X$. 

\subsubsection{Stone spaces}
\label{subsubsec:prelims_topology_StoneSpaces}
A nonempty topological space $X$ is called \emph{totally disconnected} if the only connected subsets are the singletons \cite[Page~101]{illanes1999hyperspaces}.
It is called \emph{zero-dimensional} if there exists a base of clopen sets for the topology \cite[Page~64]{illanes1999hyperspaces}. 
A compact Hausdorff space is totally disconnected if and only if it is zero-dimensional \cite[Theorem 12.11]{illanes1999hyperspaces}.
A \emph{Stone} space is a totally disconnected compact Hausdorff space. 
Note that nonempty and closed subsets of Stone spaces are Stone spaces. 

\subsubsection{Hyperspaces}
\label{subsubsec:prelims_topology_Hyperspaces}
Let $X$ be a compact Hausdorff space.
We denote $\hyper(X)$ for the set of all closed nonempty subsets of $X$.
For a finite family  $\mathcal{V}$ of open subsets of $X$ we denote 
$\langle \mathcal{V}\rangle$ for the set of all $A\in \hyper(X)$, which intersect all $V\in \mathcal{V}$ and are contained in $\bigcup \mathcal{V}=\bigcup_{V\in \mathcal{V}} V$. 
The set $\{\langle \mathcal{V} \rangle;\, \mathcal{V}\}$, where $\mathcal{V}$ ranges over all finite collections of open subsets of $X$, is a base for a compact Hausdorff topology on $\hyper(X)$, called the \emph{Vietoris topology}. 
Note that $X$ can be identified with the compact subspace of singletons $\finite_1(X)\subseteq \hyper(X)$ via $x\mapsto\{x\}$. 

Whenever $\epsilon\in \mathbb{U}_X$ is symmetric we denote $B_\epsilon(A):=\bigcup_{x\in A} B_\epsilon(x)$ for $A\in \hyper(X)$ and 
 \begin{align*}
    \eta_\epsilon
    :=\{(A,A')\in \hyper(X)^2;\, A\subseteq B_\epsilon(A'), A'\subseteq B_\epsilon(A)\}. 
\end{align*}
The set $\{\eta_\epsilon;\, \epsilon\in \mathbb{U}_X\}$ is a base for the uniformity $\mathbb{U}_{\hyper(X)}$. For details see 
\cite[Theorem~3.3]{michael1951topologies} 
and \cite[Section~8.5]{engelking1989general}. 
Whenever $X$ is metrizable, then the \emph{Hausdorff metric} induces the Vietoris topology on $\hyper(X)$. We thus observe that $X$ is metrizable if and only if $\hyper(X)$ is metrizable. 
See \cite[Section~I.2]{illanes1999hyperspaces} for details. 

\subsubsection{Regular Borel probability measures}
\label{subsubsec:prelims_topology_RegularBorelProbabilityMeasures}
Whenever $X$ is a compact Hausdorff space we equip $X$ with the Borel $\sigma$-algebra. 
We denote $C(X)$ for the set of continuous maps $X\to \mathbb{R}$. 
We denote $\mathcal{M}(X)$ for the set of all regular Borel probability measures and equip $\mathcal{M}(X)$ with the weak-* topology inherited from $C(X)^*$ via the Riesz–Markov–Kakutani theorem. 
It follows from the Banach-Alaoglu theorem that $\mathcal{M}(X)$ is compact. 
For a finite subset $\mathfrak{f}\subseteq C(X)$ and $\epsilon>0$ we denote 
\[\eta_{(\mathfrak{f},\epsilon)}:=\left\{(\mu,\nu)\in \mathcal{M}(X)^2;\, \max_{f\in \mathfrak{f}}|\mu(f)-\nu(f)|\leq\epsilon\right\}.\]
Note that such sets form a base for the uniformity $\mathbb{U}_{\mathcal{M}(X)}$.

\subsection{Groups}
\label{subsec:prelims_groups}
Let $G$ be a group. 
For subsets $M$ and $M'$ of $G$ we denote 
$MM':=\{gg';\, g\in M,g'\in M'\}$. 
Similarly we define $gM$ and $Mg$ for $g\in G$. 
We denote $\grouphull{M}$ for the subgroup generated by $M$.
For subgroups $\Gamma$ and $\Lambda$ of $G$ we denote $\Gamma\leq \Lambda$ if $\Gamma$ is a subgroup of $\Lambda$. 
For a group homomorphism $\phi\colon G\to H$ we denote $\ker(\phi):=\{g\in G;\, \phi(g)=e_H\}$ for its \emph{kernel}.

A subgroup $\Gamma\leq G$ is said to be of finite index if the partition
$G/\Gamma:=\{g\Gamma;\, g\in G\}$ is finite.
We denote $[G\colon \Gamma]:=|G/\Gamma|$ for its \emph{index (in $G$)}. 
We denote $\subfin(G)$ for the set of all finite index subgroups of $G$.
Note that $\subfin(G)$ is \emph{upward closed}, i.e.\ that for $\Gamma\in \subfin(G)$ and $\Lambda\leq G$ with $\Gamma\leq \Lambda$ we have $\Lambda\in \subfin(G)$. 

For $\Gamma\leq G$ a subset $F\subseteq G$ is called a \emph{fundamental domain} if each $g\in G$ has a unique representation as a product $g=h\gamma$ with $(h,\gamma)\in F\times \Gamma$. Note that $F$ can be chosen to contain the identity element $e_G$. 
Whenever $\Gamma\in \subfin(G)$ we have $|F|=[G\colon \Gamma]<\infty$. 

\begin{remark}
\label{rem:fundamentalDomainsComposition}
    Whenever $F$ is a fundamental domain of $\Lambda\leq G$ and $F'$ is a fundamental domain of $\Gamma\leq \Lambda$, then $F'F$ is a fundamental domain of $\Gamma\leq G$. 
    Thus, for $\Gamma,\Lambda\in \subfin(G)$ with $\Gamma\leq \Lambda$ we have $[G\colon \Lambda][\Lambda \colon \Gamma]=[G\colon \Gamma]$. 
    In particular, whenever $\Gamma\leq \Lambda$ have the same index in $G$ we have $[\Lambda\colon \Gamma]=1$, i.e.\ $\Gamma=\Lambda$. 
\end{remark}

For $\Gamma\leq G$ and $g\in G$ we denote\footnote{
Note that our definition of $\Gamma^g$ differs from that in \cite{ribes2000profinite}, where $\Gamma^g$ denotes $g^{-1}\Gamma g$. Our convention is chosen so that it interacts more naturally with left actions. This will become clear in the discussion below.
} 
$\Gamma^g:=g\Gamma g^{-1}$ for the \emph{conjugate of $\Gamma$ by $g$}.
Note that $\Gamma^g\leq G$ and that $(\Gamma^g)^h=\Gamma^{(hg)}$ for $g,h\in G$. 
Two subgroups $\Gamma$ and $\Lambda$ are called \emph{conjugated} if there exists $g\in G$ with $\Gamma=\Lambda^g$. 
A subgroup $\Gamma\leq G$ is called \emph{normal (in $G$)} if $\Gamma^g=\Gamma$ for all $g\in G$. 
The \emph{normal core} of $\Gamma\leq G$ is given by $\Gamma_G:=\bigcap_{g\in G} \Gamma^g$. 
It is the largest subgroup of $\Gamma$ that is normal in $G$.
For further details we recommend \cite{ribes2000profinite}.

\subsection{Actions}
\label{subsec:prelims_actions}
Let $G$ be a group and $X$ be a compact Hausdorff space. 
An \emph{action} of $G$ on $X$ is a group homomorphism $\alpha$ from $G$ into the group of homeomorphisms $X\to X$.
For $g\in G$ and $x\in X$, we suppress the symbol for the action by simply writing $g.x:=\alpha(g)(x)$. 
This allows us to simply speak of an \emph{action $(X,G)$}. 
An action $(X,G)$ is called \emph{finite} if $X$ is finite.
See \cite{auslander1988minimal} for a reference for the following and more details on actions.  

Let $(X,G)$ be an action. 
For $x\in X$ we denote $G.x:=\{g.x; g\in G\}$ for the \emph{orbit} of $x$. An action $(X,G)$ is called \emph{minimal} if all $x\in X$ have a dense orbit. 
For $x\in X$ we denote $G_0(x):=\{g\in G;\, g.x=x\}$ for the \emph{stabilizer of $x$}. 
Note that $G_0(g.x)=G_0(x)^g$ holds for all $g\in G$.
For $\epsilon\in \mathbb{U}_X$ we denote $G_\epsilon(x):=\{g\in G;\, (g.x,x)\in \epsilon\}$. 

If $(X_i,G)_{i\in I}$ is a family of actions the \emph{product} $\prod_{i\in I} X_i$ is also a compact Hausdorff space. 
On $\prod_{i\in I} X_i$ we consider the action given by $g.(x_i)_{i\in I}:=(g.x_i)_{i\in I}$. 
For a subset $M\subseteq X$ and $g\in G$ we denote $g.M:=\{g.x;\, x\in M\}$. 
A subset $M\subseteq X$ is called \emph{invariant} if $g.M=M$ holds for all $g\in G$. 
A subset of $X^2$ is called \emph{invariant} if it is invariant w.r.t.\ $(X^2,G)$. 
If $A\subseteq X$ is a closed nonempty and invariant subset, then 
$G\times A\to A$ establishes an action on $A$, the \emph{subaction on $A$}. 

A pair $(x,x')\in X^2$ is called \emph{proximal} if for all $\epsilon\in \mathbb{U}_X$ there exists $g\in G$ such that $g.(x,x')\in \epsilon$. 
It is called \emph{distal} if it is not proximal. 
An action $(X,G)$ is called distal if all pairs in $X^2$ are distal.
See \cite{auslander1988minimal} for further reading.

\subsubsection{Equivariant maps}
\label{subsubsec:prelims_actions_EquivariantMaps}
A continuous mapping $\pi\colon X\to Y$ between actions $(X,G)$ and $(Y,G)$ is called \emph{equivariant} if $\pi(g.x)=g.\pi(x)$ holds for all $g\in G$ and $x\in X$. 
An equivariant homeomorphism is called a \emph{conjugacy}. 
A conjugacy $X\to X$ is called an \emph{automorphism}. 
Two actions are called \emph{conjugated} if there exists a conjugacy between them. 
An equivariant surjection is called a \emph{factor map}. 
$(X,G)$ is called an \emph{extension} of $(Y,G)$ and $(Y,G)$ is called a \emph{factor} of $(X,G)$ if there exists a factor map $\pi\colon X\to Y$. 
Any factor map is \emph{closed}, i.e.\ images of closed sets are closed. 
Any equivariant map between minimal actions is a factor map. 

If $R\subseteq X$ is an invariant closed equivalence relation, then $g.R[x]:=R[g.x]$ induces an action on $X/R$ and the quotient mapping $x\mapsto R[x]$ is a factor map.  
Furthermore, for a factor map $\pi\colon X\to Y$ we denote $R(\pi):=\{(x,x')\in X^2;\, \pi(x)=\pi(x')\}$ for the \emph{fibre relation}. 
$R(\pi)$ is an invariant closed equivalence relation and $(Y,G)$ is conjugated to $(X/R(\pi),G)$. 

\subsubsection{Equicontinuity}
\label{subsubsec:prelims_actions_Equicontinuity}
An action $(X,G)$ is called \emph{equicontinuous} if, for any $\epsilon\in \mathbb{U}_X$ there exists $\delta\in \mathbb{U}_X$ such that for all $(x,x')\in \delta$ and $g\in G$ we have $(g.x,g.x')\in \epsilon$. 
Note that $(X,G)$ is equicontinuous if and only if $\mathbb{U}_X$ admits a base consisting of invariant entourages. 
Minimal equicontinuous actions $(X,G)$ are \emph{coalescent}, i.e.\ any factor map $\pi\colon X\to X$ is a conjugacy \cite[Page~81]{auslander1988minimal}. In particular, minimal equicontinuous actions $(X,G)$ and $(Y,G)$ are conjugated, if and only if there exist factor maps $X\to Y$ and $Y\to X$. 
A minimal equicontinuous action $(X,G)$ is called \emph{regular}\footnote{
    Note that for minimal equicontinuous actions $(X,G)$ also $(X^2,G)$ is equicontinuous. 
    Thus for minimal equicontinuous actions the notion of regularity from \cite[Page~147]{auslander1988minimal} can be equivalently formulated as we did it.}, 
if for $x,x'\in X$ there exists an automorphism $\iota\colon X\to X$ with $\iota(x)=x'$.

\subsubsection{Rotations}
\label{subsubsec:prelims_actions_Rotations}
Let $G$ be a group and $X$ be a compact (Hausdorff) group.  
A \emph{(group) rotation} is an action of $G$ on $X$ of the form
$g.x = \phi(g) x$,
where $\phi \colon G \to X$ is a group homomorphism. 
Any rotation is equicontinuous. 
A rotation is minimal if and only if $\phi(G)$ is dense in $X$.
Note that for a minimal rotation, $\phi$ is uniquely determined by $\phi(g)=g.e_X$, where $e_X$ denotes the identity of $X$. 
Any minimal rotation $(X,G)$ is regular, since for $x,x'\in X$ we have that $y\mapsto yx^{-1}x'$ yields a conjugacy $\iota\colon X\to X$ with $\iota(x)= x'$.

\subsubsection{Induced dynamics}
\label{subsubsec:prelims_actions_InducedDynamics}
Let $(X,G)$ be an action. 
For $A\in \hyper(X)$ and $g\in G$ we denote $g.A:=\{g.x;\, x\in A\}$. 
This induces an action $(\hyper(X),G)$. 
Furthermore, for $g\in G$ and $f\in C(X)$ we denote $g^*f:=f\circ g$, where we identify $g$ with the homeomorphism induced by $g$. 
For $\mu\in \mathcal{M}(X)$ we denote $g.\mu(f):=\mu(g^*f)$. 
Note that $g.\mu\in \mathcal{M}(X)$. 
This defines an action $(\mathcal{M}(X),G)$. 
A measure $\mu\in \mathcal{M}(X)$ is called \emph{invariant} if $g.\mu=\mu$ holds for all $g\in G$. 
Any minimal equicontinuous action allows for a unique invariant regular Borel probability measure \cite[Chapter~7]{auslander1988minimal}. 
The following proposition is well known. See for example \cite[Proposition 7]{bauer1975topological} for the case of actions of $\mathbb{Z}$ on compact metrizable spaces. We include the short argument for the convenience of the reader. 

\begin{proposition}
    For an equicontinuous action $(X,G)$ the actions $(\hyper(X),G)$ and $(\mathcal{M}(X),G)$ are equicontinuous. 
\end{proposition}
\begin{proof}
    Recall from \cite[Proposition~2.1]{auslander1988minimal} and \cite[Theorem~3.2]{auslander1988minimal} that an action $(X,G)$ is equicontinuous if and only if there exists a compact group $H$ and a continuous\footnote{An action of a topological group $G$ is called continuous if $G\times X\to X$ is continuous.} action $(X,H)$ that extends $(X,G)$. 
    It is straightforward to observe that $(\hyper(X),H)$ is a continuous action that extends $(\hyper(X),G)$. It follows that $(\hyper(X),G)$ is equicontinuous. A similar argument also yields the statement for $(\mathcal{M}(X),G)$. 
\end{proof}

\subsubsection{Maximal equicontinuous factors}
\label{subsubsec:prelims_actions_MaximalEquicontinuousFactor}
For any action $(X,G)$ there exists a \emph{maximal equicontinuous factor}, i.e.\ a factor map $\pi_{\operatorname{MEF}}\colon X\to X_{\operatorname{MEF}}$ onto an equicontinuous action $(X_{\operatorname{MEF}},G)$ such that for any factor map $\pi\colon X\to Y$ onto an equicontinuous action $(Y,G)$ there exists a factor map $\psi\colon X_{\operatorname{MEF}}\to Y$ with $\pi=\psi \circ \pi_{\operatorname{MEF}}$. The maximal equicontinuous factor is unique up to conjugacy. For details see \cite{auslander1988minimal}. 

\subsubsection{Eigenvalues}
\label{subsubsec:prelims_actions_Eigenvalues}
Let $(X,G)$ be a minimal action. 
A finite index subgroup $\Gamma\leq G$ is called an \emph{eigenvalue} of $(X,G)$ if $G/\Gamma$ is a factor of $(X,G)$. We denote $\Eig(X,G)$ for the set of all \emph{eigenvalues} of $(X,G)$. 
For $x\in X$ we denote $\Eig_x(X,G)$ for the set of all $\Gamma\in \subfin(G)$ for which there exists a factor map $\pi\colon X\to G/\Gamma$ with $\pi(x)=\Gamma$.

Consider a subset $S\subseteq \subfin(G)$. 
$S$ is called a \emph{scale} if for $\Gamma,\Gamma'\in S$ there exists $\Lambda\in S$ with $\Lambda\subseteq \Gamma\cap \Gamma'$. 
$S$ is called \emph{intersection closed} if for $\Gamma,\Lambda\in S$ also $\Gamma\cap \Lambda\in S$. 
$S$ is called \emph{upward closed} if for all $\Gamma\in S$ and all $\Lambda\in \subfin(G)$ with $\Gamma\subseteq \Lambda$ it follows that $\Lambda\in S$. 
$S$ is called a \emph{filter} if $S$ is an upward closed scale. 
$S$ is called \emph{conjugation invariant} if for $\Gamma\in S$ and $g\in G$ also $\Gamma^g\in S$.
$S$ is called \emph{core-stable} if for $\Gamma\in S$ also $\Gamma_G\in S$. 
Note that any filter is intersection closed. 

For $S\subseteq \subfin(G)$ we denote $\eigenhull{S}$ for the set of all $\Lambda\in \subfin(G)$ such that there exists $\Gamma\in S$ and $g\in G$ with $\Gamma^g\subseteq \Lambda$. $\eigenhull{S}$ is the smallest conjugation invariant and upward closed subset of $\subfin(G)$ that contains $S$ \cite[Remark~5.12]{cortez2026minimal}. 
For any minimal action $\Eig(X,G)$ is conjugacy invariant and upward closed. 
Furthermore, for $x\in X$ the set $\Eig_x(X,G)$ is a filter that satisfies $\eigenhull{\Eig_x(X,G)}=\Eig(X,G)$ \cite[Theorem 4.7 and Proposition~5.13]{cortez2026minimal}.

\subsection{Subodometers}
\label{subsec:prelims_subodometers}. 
We next collect some results on subodometers. 
For further details see \cite{cortez2026minimal}. 

\begin{proposition}\cite[Proposition~5.1 and Theorem~5.5]{cortez2026minimal}
\label{pro:EIGfactorsAndConjugacyInvariant}
    \begin{itemize}
        \item[(i)] Let $(X,G)$ be a minimal action. 
        A subodometer $(Y,G)$ is a factor of $(X,G)$ if and only if $\Eig(Y,G)\subseteq \Eig(X,G).$
        \item[(ii)] Two subodometers are conjugated if and only if they have the same eigenvalues. 
        \item[(iii)] Any factor of a subodometer is a subodometer. 
    \end{itemize}     
\end{proposition}

\begin{proposition}\cite[Proposition~3.2]{cortez2026minimal} 
\label{pro:characterizationEquicontinuousStone} 
    A minimal action $(X,G)$ is a subodometer if and only if $\mathbb{U}_X$ allows for a base consisting of invariant equivalence relations. 
\end{proposition}

\subsubsection{Finite subodometers}
\label{subsubsec:prelims_subodometers_FiniteSubodometers} 
Any finite minimal action is a subodometer.
For finite actions $(X,G)$ and $x\in X$ we have $G_0(x)\in \subfin(G)$. 

Let $(X,G)$ and $(Y,G)$ be finite subodometers and $(x,y)\in X\times Y$. It is straightforward to show that there exists a factor map $\pi\colon X\to Y$ with $\pi(x)=y$ if and only if $G_0(x)\subseteq G_0(y)$. 
Furthermore, there exists a conjugacy $\iota\colon X\to Y$ with $\iota(x)=y$ if and only if $G_0(x) = G_0(y)$. 

For $\Gamma\in \subfin(G)$ we denote $G/\Gamma$ for the finite partition $\{h\Gamma;\, h\in G\}$ and consider the induced action $(G/\Gamma,G)$ given by $g.(h\Gamma):=gh\Gamma$ for $g\in G$ and $h\Gamma\in G/\Gamma$. 
Clearly, we have $G_0(\Gamma)=\Gamma$.
Thus, whenever $(X,G)$ is a finite subodometer and $x\in X$, then $(X,G)$ and $(G/G_0(x),G)$ are conjugated and the conjugacy can be chosen such that $x\mapsto G_0(x)$. 
Note that $\Gamma$ is normal if and only if $(G/\Gamma,G)$ is an odometer. 

\subsubsection{Scales for subodometers}
\label{subsubsec:prelims_subodometers_ScalesForSubodometers}
Let $S\subseteq \subfin(G)$ be a scale.
Note that for $\Gamma_1,\Gamma_2\in S$ with $\Gamma_1\subseteq \Gamma_2$ there exists a unique factor map $\pi_{\Gamma_2}^{\Gamma_1}\colon G/\Gamma_1 \to G/\Gamma_2$ with $\Gamma_1\mapsto \Gamma_2$. 
For $\Gamma_1,\Gamma_2,\Gamma_3\in \subfin(G)$ with $\Gamma_1\subseteq \Gamma_2\subseteq \Gamma_3$ we have $\pi_{\Gamma_3}^{\Gamma_1}=\pi_{\Gamma_3}^{\Gamma_2}\circ \pi_{\Gamma_2}^{\Gamma_1}$. 
Thus, any scale $S$ induces an inverse system $(\pi_{\Lambda}^{\Gamma})_{(\Gamma,\Lambda)\in S_*^2}$ of factor maps, where we abbreviate 
$S_*^2:=\{(\Gamma,\Lambda)\in S^2;\, \Gamma\subseteq \Lambda\}$.
The \emph{$S$-subodometer} is given by the inverse limit 
    \[
    \varprojlim_{\Gamma \in S} G\big/\Gamma 
    := 
    \Bigl\{ (x_\Gamma)_{\Gamma \in S} \in \prod_{\Gamma \in S} G\big/\Gamma 
    ;\, 
    \pi_{\Lambda}^\Gamma(x_\Gamma) = x_{\Lambda} \text{ for all } (\Gamma,\Lambda)\in S_*^2
    \Bigr\}.
    \] 
Note that the properties of minimality, total disconnectedness and equicontinuity are preserved under inverse limits.
Thus the $S$-subodometer is a subodometer. 
For a scale $S$ we say that a subodometer $(X,G)$ is \emph{generated by $S$} if it is conjugated to the $S$-subodometer.     
See \cite[Section 1.1]{ribes2000profinite} for further details on inverse systems and inverse limits. 

Any subodometer $(X,G)$ can be represented as such an inverse limit. 
As presented in \cite[Theorem~4.7]{cortez2026minimal} for any choice $x\in X$ the subodometer $(X,G)$ is generated by the scale $\Eig_x(X,G)$.
Furthermore, a scale $S$ generates a subodometer $(X,G)$ if and only if $\Eig(X,G)=\eigenhull{S}$, i.e.\ if for all $\Gamma\in \Eig(X,G)$ there exists $\Lambda\in S$ and $g\in G$ with $\Lambda\subseteq \Gamma^g$ \cite[Corollary~5.15]{cortez2026minimal}.

\subsubsection{Odometers}
\label{subsubsec:prelims_subodometers_Odometers}

\begin{proposition}\cite[Theorem~7.6]{cortez2026minimal}
\label{pro:odometerCharacterizationEIG}
    For a subodometer $(X,G)$ the following statements are equivalent. 
    \begin{itemize}
        \item[(i)] $(X,G)$ is an odometer. 
        \item[(ii)] $(X,G)$ is generated by a scale consisting of normal subgroups of $G$. 
        \item[(iii)] $\Eig(X,G)$ is a scale\footnote{Note that $\Eig(X,G)$ is always upward closed and hence $\Eig(X,G)$ is a scale if and only if it is intersection closed if and only if it is a filter.}.
        \item[(iv)] $\Eig(X,G)$ is core-stable.
    \end{itemize}
\end{proposition}

\subsubsection{The universal odometer}
\label{subsubsec:prelims_subodometers_TheUniversalOdometer}
There exists a \emph{universal odometer} $(X,G)$ that has all subodometers $(Y,G)$ as factors \cite[Section~9]{cortez2026minimal}. 
It satisfies $\Eig(X,G)=\subfin(G)$ and is generated by the scale $S$ consisting of all normal $\Gamma\in \subfin(G)$. 

\subsubsection{The enveloping odometer}
\label{subsubsec:prelims_subodometers_TheEnvelopingOdometer} 
For any subodometer $(X,G)$ there exists an odometer $(\envelopingOdometer{X},G)$ that is an extension of $(X,G)$ and such that all odometers that are extensions of $(X,G)$ are also extensions of $(\envelopingOdometer{X},G)$. 
This odometer is unique up to conjugacy and called the \emph{enveloping odometer}. It is conjugated to the action of $G$ on the Ellis-semigroup.  
For details see \cite[Section~9]{cortez2026minimal} and \cite[Chapter~3]{auslander1988minimal}.

\section{Regular recurrence}
\label{sec:RegularRecurrence}
A point $x\in X$ is called \emph{regularly recurrent} if for any open neighbourhood $U$ of $x$ there exists a finite index subgroup $\Gamma\leq G$ such that $\Gamma.x\subseteq U$. 
Note that by restricting to the core we can always choose $\Gamma$ to be normal. 
An action $(X,G)$ is called \emph{pointwise regularly recurrent} if all $x\in X$ are regularly recurrent. 

\subsection{Almost periodicity}
\label{subsec:RegularRecurrence_almostPeriodicity}
We next summarize some aspects of almost periodicity. For further details see \cite{auslander1988minimal}. 
A subset $S\subseteq G$ is called \emph{syndetic} if there exists a finite set $F\subseteq G$ with $FS=G$. 
Let $(X,G)$ be an action. 
A point $x\in X$ is called \emph{almost periodic} if for any open neighbourhood $U$ of $x$ there exists a syndetic subset $S\subseteq G$ such that $S.x\subseteq U$. 
Clearly, any regularly recurrent point is almost periodic. 
A point $x\in X$ is almost periodic if and only if $(\overline{G.x},G)$ is minimal \cite[Theorem~1.7]{auslander1988minimal}. 

An action $(X,G)$ is called \emph{pointwise almost periodic} (also \emph{semi-simple}) if all $x\in X$ are almost periodic, i.e.\ if $X$ decomposes into minimal components. Any distal action is pointwise almost periodic \cite[Corollary~5.4]{auslander1988minimal}. 
An action is called \emph{uniformly almost periodic} if for all $\epsilon\in \mathbb{U}_X$ there exists a syndetic subset $S\subseteq G$ such that for all $x\in X$ we have $S.x\subseteq B_\epsilon(x)$. 
Clearly, any uniformly almost periodic action is pointwise almost periodic. 
Recall from \cite[Theorem~2.2]{auslander1988minimal} that an action is equicontinuous if and only if it is uniformly almost periodic. 

\begin{lemma}\cite[Lemma~1.12]{auslander1988minimal}
\label{lem:normalSubgroupMinimalComponents}
    Let $(X,G)$ be an action and $x\in X$. 
    Let $\Gamma\leq G$ be a normal subgroup. 
    If $x$ is an almost periodic point for $(X,G)$, then for all $g\in G$ the point $g.x$ is almost periodic for $(X,\Gamma)$.
\end{lemma}

\subsection{$G$-tiles}
\label{subsec:RegularRecurrence_GTiles}
Let $(X,G)$ be an action. 
An open subset $A\subseteq X$ is called a \emph{$G$-tile} if $\{g.A;\, g\in G\}$ is a partition of $X$. Note that the compactness of $X$ yields that the partition given by a $G$-tile is finite and hence that $G$-tiles are clopen. $G$-tiles are closely related to regular recurrence as illustrated by the following. 

\begin{proposition}
\label{pro:regularlyRecurrentPointsAndGTiles}
    Let $(X,G)$ be a minimal action.
    A point $x\in X$ is regularly recurrent if and only if it allows for a neighbourhood base consisting of $G$-tiles. 
\end{proposition}
\begin{proof}
'$\Rightarrow$': 
    Let $U$ be an open neighbourhood of $x$. 
    Choose an open neighbourhood $V$ and a closed neighbourhood $C$ of $x$ such that $x\in V\subseteq C \subseteq U$. 
    Since $x$ is regularly recurrent there exists a normal finite index subgroup $\Gamma\leq G$ with $\Gamma.x\subseteq V$. 
    In particular, we observe that $A:=\overline{\Gamma.x}\subseteq C\subseteq U$. 
    It remains to show that $A$ is a $G$-tile. 
    Note that $A$ is closed. Thus to establish $A$ as a $G$-tile it suffices to show that $\{g.A;\, g\in G\}$ is a finite partition. 

    We first show that $\{g.A;\, g\in G\}$ is a partition.     
    Since $x\in X$ is regularly recurrent it is almost periodic and hence for any $g\in G$ we have that $g.x$ is an almost periodic point for $(X,\Gamma)$. 
    It follows from 
    $
        \overline{\Gamma g.x}
        =\overline{g\Gamma.x}
        =g.\overline{\Gamma.x}
        = g.A
    $
    that $(g.A,\Gamma)$ is minimal. Since $(X,G)$ is minimal we observe $\{g.A;\, g\in G\}$ to be the minimal components of $X$. In particular, $\{g.A;\, g\in G\}$ is a partition. 
    To show that it is finite consider a fundamental domain $F$ for $\Gamma$. 
    We have 
    \begin{align*}
        \bigcup_{g\in F}g.A
        = \bigcup_{g\in F}g.\overline{\Gamma.x}
        = \bigcup_{g\in F}\overline{g\Gamma.x}
        = \overline{\bigcup_{g\in F}g\Gamma.x}
        = \overline{F\Gamma.x}
        = \overline{G.x}
        =X,
    \end{align*}
    which establishes $\{g.A;\, g\in G\}$ as a finite partition. 

'$\Leftarrow$':
    Let $U$ be an open neighbourhood of $x$. 
    Since $x$ allows for a neighbourhood base of $G$-tiles we find a $G$-tile $A$ with $x\in A\subseteq U$. 
    Denote $R:=\bigcup_{g\in G}g.A^2$ for the respective equivalence relation and note that $R$ is an invariant closed equivalence relation and that $Y:=X/R$ is finite. Denote $\pi\colon X\to Y$ for the respective factor map and note that $\Gamma:=G_0(\pi(x))$ is a finite index subgroup of $G$. 
    For $g\in \Gamma$ we observe that $\pi(g.x)=\pi(x)$ and hence that $g.x\in A\subseteq U$. This shows $\Gamma.x\subseteq U$.
\end{proof}

Since $G$-tiles are clopen we observe that only Stone spaces allow for minimal pointwise regularly recurrent actions. 
From \cite[Theorem~3.6]{cortez2026minimal} we recall the following. 

\begin{proposition} \cite[Theorem~3.6]{cortez2026minimal}
\label{pro:characterizationGTiles} 
    A minimal action $(X,G)$ is a subodometer if and only if the topology of $X$ allows for a base of $G$-tiles.  
\end{proposition}

\begin{corollary}
\label{cor:characterizationSubodometersViaRegularRecurrence}
    An action is pointwise regularly recurrent if and only if it decomposes into a disjoint union of subodometers. 
    In particular, a minimal action is a subodometer if and only if it is pointwise regularly recurrent. 
\end{corollary}
\begin{proof}
    If $(X,G)$ is pointwise regularly recurrent, then it is pointwise almost periodic and hence decomposes into minimal components. 
    Such a minimal component $Y\subseteq X$ is then pointwise regularly recurrent and Proposition \ref{pro:characterizationGTiles} yields that the topology of $Y$ has a base consisting of $G$-tiles. 
    From Proposition \ref{pro:regularlyRecurrentPointsAndGTiles} we observe that $(Y,G)$ is a subodometer. 

    For the converse assume that $X$ decomposes into subodometers. 
    By Proposition \ref{pro:regularlyRecurrentPointsAndGTiles} any minimal component has a base consisting of $G$-tiles and Proposition \ref{pro:characterizationGTiles} yields that it consists of regularly recurrent points. Thus, $X$ consists of regularly recurrent points. 
\end{proof}

\subsection{Uniform regular recurrence}
\label{subsec:RegularRecurrence_UniformRegularRecurrence}
An action $(X,G)$ is called \emph{uniformly regularly recurrent} if for any $\epsilon\in \mathbb{U}_X$ there exists a finite index subgroup $\Gamma\leq G$ such that for all $x\in X$ we have $\Gamma.x\subseteq B_\epsilon(x)$. 

\begin{example}
    A rational rotation on the circle is uniformly regularly recurrent.
\end{example}

Clearly, any uniformly regularly recurrent action is pointwise regularly recurrent, and uniformly almost periodic, i.e.\ equicontinuous. We next show that also the converse holds. 

\begin{proposition}
\label{pro:uniformRegularRecurrenceVSPointwiseRegularRecurrence}
    An action $(X,G)$ is uniformly regularly recurrent if and only if it is pointwise regularly recurrent and equicontinuous. 
\end{proposition}
\begin{proof}
    It remains to show that any pointwise regularly recurrent and equicontinuous action is uniformly regularly recurrent. 
    Let $\epsilon\in \mathbb{U}_X$. 
    Since $(X,G)$ is equicontinuous there exists an invariant and symmetric $\delta\in \mathbb{U}_X$ with $\delta\delta\delta\subseteq \epsilon$. 
    Since $X$ is compact we find a finite subset $F\subseteq X$ with $X=\bigcup_{x\in F}B_\delta(x)$. 
    For $x\in F$ there exists a finite index subgroup $\Gamma_x$ such that $\Gamma_x.x\subseteq B_\delta(x)$. 
    For the finite index subgroup $\Gamma:=\bigcap_{x\in F}\Gamma_x$ we observe that for all $x\in F$ we have $\Gamma.x\subseteq B_\delta(x)$. 

    Let $x\in X$ and $g\in \Gamma$. 
    Choose $x'\in F$ with $x\in B_\delta(x')$, i.e.\ $(x',x)\in \delta$. 
    From $g.x'\in \Gamma.x'\subseteq B_\delta(x')$ we know that $(g.x',x')\in \delta$. 
    Since $\delta$ is invariant we have $(g.x,g.x')\in \delta$. 
    Combining these observations we observe 
    $(g.x,x)\in \delta\delta\delta\subseteq \epsilon$, i.e.\ $g.x\in B_\epsilon(x)$. This shows $\Gamma.x\subseteq B_\epsilon(x)$ for all $x\in X$. 
\end{proof}

Recall from Corollary \ref{cor:characterizationSubodometersViaRegularRecurrence} that a minimal action is a subodometer if and only if it is pointwise regularly recurrent. 
Since subodometers are equicontinuous we conclude that for minimal actions pointwise and uniform regular recurrence are equivalent. 

\begin{corollary}
\label{cor:subodometersUniformlyRegularlyRecurrent}
    A minimal action is a subodometer if and only if it is uniformly regularly recurrent. 
\end{corollary}

The following example illustrates that beyond minimality a pointwise regularly recurrent action does not need to be uniformly regularly recurrent. 

\begin{example}
    In this example we denote $\overline{\mathbb{N}}:=\mathbb{N}\cup \{\infty\}$ and use the convention $1/\infty:=0$. 
    For $n\in \overline{\mathbb{N}}$ we denote 
    $X_{n}:=\{z\in \mathbb{C};\, |z|=1+1/n\}$ for the circle centered at $0$ with radius $1+1/n$.
    On $X_n$ we consider the rotation by $1/n$, i.e.\ the action of $\mathbb{Z}$ given by 
    $g.x:=e^{2 \pi i g/n}x$. Note that $(X_n,\mathbb{Z})$ consists of regularly recurrent points. 
    Furthermore, note that $(X_\infty,\mathbb{Z})$ consists of fixed points. 

    Denote $X:=\bigcup_{n\in \overline{\mathbb{N}}}X_n$. 
    Note that for $n\to \infty$ the rotations $X_n$ slow down sufficiently. 
    Thus, we obtain an action $(X,\mathbb{Z})$. 
    Clearly, $(X,\mathbb{Z})$ consists of regularly recurrent points. 
    To observe that $(X,\mathbb{Z})$ is not equicontinuous consider $\delta>0$. There exists $n\in 2\mathbb{N}$ with $1/n<\delta$.
    Denote $x:=1$ and $x':=1+1/n$ and note that $d(x,x')<\delta$. 
    Nevertheless, for $g:=n/2$ we have $d(g.x,g.x')=d(1,-1-1/n)>1$. 
    Thus, $(X,\mathbb{Z})$ is not equicontinuous. 
    Since any uniformly regularly recurrent action is uniformly almost periodic, i.e.\ equicontinuous, we observe that $(X,\mathbb{Z})$ is not uniformly regularly recurrent. 
\end{example}

As witnessed by irrational rotations on the circle equicontinuity does not necessarily imply the existence of regularly recurrent points. 
For actions on Stone spaces we have the following. 

\begin{corollary}
\label{cor:uniformRegularRecurrenceAndStoneSpaces}
    An action $(X,G)$ on a Stone space is equicontinuous if and only if it is uniformly regularly recurrent. 
\end{corollary}
\begin{proof}
    It remains to show that equicontinuous actions on Stone spaces are uniformly regularly recurrent. 
    Any equicontinuous action is pointwise almost periodic and hence decomposes into minimal components. These components are minimal equicontinuous actions on Stone spaces, i.e.\ subodometers and Corollary \ref{cor:characterizationSubodometersViaRegularRecurrence} yields that they consist of regularly recurrent points. This shows that $(X,G)$ is pointwise regularly recurrent. Since $(X,G)$ is equicontinuous Proposition \ref{pro:uniformRegularRecurrenceVSPointwiseRegularRecurrence} yields that $(X,G)$ is uniformly regularly recurrent.
\end{proof}

\subsection{Uniform regular recurrence of induced actions}
\label{subsec:RegularRecurrence_UniformRegularRecurrenceOfInducedActions}
Next, we show that uniform regular recurrence of an action $(X,G)$ is inherited by the induced actions $(\hyper(X),G)$ and $(\mathcal{M}(X),G)$.

\begin{lemma}
\label{lem:regularRecurrenceInCX}
    Let $(X,G)$ be a uniformly regularly recurrent action and $\epsilon>0$.
    For $f\in C(X)$ there exists a finite index subgroup $\Gamma\leq G$ such that for all $g\in \Gamma$ we have 
    $\|f-g^*f\|_\infty\leq \epsilon.$
\end{lemma}
\begin{proof}
    Since $X$ is compact $f$ is uniformly continuous and hence there exists $\eta\in \mathbb{U}_X$ such that for $(x,x')\in \eta$ we have 
    $|f(x)-f(x')|<\epsilon$. 
    Choose a finite index subgroup $\Gamma\leq G$ such that $g.x\in B_\eta(x)$ holds for all $(g,x)\in \Gamma\times X$. 
    
    Consider $g\in \Gamma$. 
    For $x\in X$ we have $(g.x,x)\in \eta$ and hence 
    \begin{align*}
        |f(x)-g^*f(x)|
        =|f(x)-f(g.x)|
        <\epsilon. 
    \end{align*} 
    We thus observe $\|f-g^*f\|_\infty \leq \epsilon$. 
\end{proof}

\begin{theorem}
\label{the:HMinheritanceUniformRegularRecurrence}
    For an action $(X,G)$ the following statements are equivalent. 
    \begin{itemize}
        \item[(i)] $(X,G)$ is uniformly regularly recurrent. 
        \item[(ii)] $(\hyper(X),G)$ is uniformly regularly recurrent. 
        \item[(iii)] $(\mathcal{M}(X),G)$ is uniformly regularly recurrent.
    \end{itemize}
\end{theorem}
\begin{proof}
    Note that $(X,G)$ can be identified as a subaction of $(\hyper(X),G)$ via $x\mapsto \{x\}$. Furthermore, it can be identified as a subaction of $(\mathcal{M}(X),G)$ via $x\mapsto \delta_x$. 
    Since subactions of uniformly regularly recurrent actions are uniformly regularly recurrent it suffices to show that for a uniformly regularly recurrent action $(X,G)$ also $(\hyper(X),G)$ and $(\mathcal{M}(X),G)$ are uniformly regularly recurrent. 

'$\hyper(X)$':
    Consider $\eta\in \mathbb{U}_{\hyper(X)}$. 
    There exists\footnote{
        Recall that we define 
        $\eta_\epsilon
        :=\{(A,A')\in \hyper(X)^2;\, A\subseteq B_\epsilon(A'), A'\subseteq B_\epsilon(A)\}$ for symmetric $\epsilon\in \mathbb{U}_X$ and that sets of this form establish a base of $\mathbb{U}_{\hyper(X)}$.
    } a symmetric $\epsilon\in \mathbb{U}_X$ with 
    $\eta_\epsilon\subseteq \eta$. 
    Since $(X,G)$ is uniformly regularly recurrent there exists a finite index subgroup $\Gamma\leq G$ with $\Gamma.x\subseteq B_\epsilon(x)$ for all $x\in X$. 

    Let $A\in \hyper(X)$ and $g\in \Gamma$. 
    For $x\in A$ we have $g.x\in B_\epsilon(x)\subseteq B_\epsilon(A)$ and hence $g.A\subseteq B_\epsilon(A)$. 
    Since $\epsilon$ is symmetric for $x\in A$ we also know $x\in B_\epsilon(g.x)\subseteq B_\epsilon(g.A)$ and hence $A\subseteq B_\epsilon(g.A)$. 
    Thus, $(g.A,A)\in \eta_\epsilon$. 
    This shows $\Gamma.A\subseteq B_{\eta_\epsilon}(A)$ for all $A\in \hyper(X)$. 

'$\mathcal{M}(X)$': 
    Consider $\eta\in \mathbb{U}_{\mathcal{M}(X)}$. 
    There exist\footnote{
        Recall that the sets $\eta_{(\mathfrak{f},\epsilon)}:=\left\{(\mu,\nu)\in \mathcal{M}(X)^2;\, \max_{f\in \mathfrak{f}}|\mu(f)-\nu(f)|\leq\epsilon\right\}$ with $\mathfrak{f}\subseteq C(X)$ finite and $\epsilon>0$
        form a base for $\mathbb{U}_{\mathcal{M}(X)}$. 
    } a finite subset $\mathfrak{f}\subseteq C(X)$ and $\epsilon>0$ such that
    $\eta_{(\mathfrak{f},\epsilon)}\subseteq \eta$. 
    From Lemma \ref{lem:regularRecurrenceInCX} we know of the existence of finite index subgroups $\Gamma_f\subseteq G$ such that 
    $\|f-g^*f\|_\infty\leq \epsilon$ 
    holds
    for all $f\in \mathfrak{f}$ and $g\in \Gamma_f$.     
    Clearly, $\Gamma:=\bigcap_{f\in \mathfrak{f}}\Gamma_f$ is also a finite index subgroup of $G$. 
    For $\mu \in \mathcal{M}(X)$, $g\in \Gamma$ and $f\in \mathfrak{f}$ we have
    \begin{align*}
        |\mu(f)-g.\mu(f)|
        =|\mu(f-g^*f)|
        \leq \mu(|f-g^*f|)
        \leq \|f-g^*f\|_\infty\leq \epsilon
    \end{align*}
    and hence $(\mu,g.\mu)\in \eta_{(\mathfrak{f},\epsilon)}\subseteq \eta$. This shows $\Gamma.\mu \subseteq B_\eta(\mu)$ for all $\mu\in \mathcal{M}(X)$. 
\end{proof}

From Corollary \ref{cor:characterizationSubodometersViaRegularRecurrence} we observe the following. Recall that subodometers and rational rotations on the circle are examples of uniformly regularly recurrent actions. In particular, Theorem \ref{the:INTROInducedDynamicsOfSubodometerDecompose} is a special case of the following. 

\begin{corollary}
\label{cor:InducedDynamicsDecompose}
    For a uniformly regularly recurrent action $(X,G)$ the actions $(\hyper(X),G)$ and $(\mathcal{M}(X),G)$ decompose into a disjoint union of subodometers. 
\end{corollary}

\begin{remark}
    If $(X,G)$ is minimal, then $(\hyper(X),G)$ is pointwise regularly recurrent if and only if $(\hyper(X),G)$ is uniformly regularly recurrent. 
    Indeed, if $(\hyper(X),G)$ is pointwise regularly recurrent, then so is $(X,G)$ and hence $(X,G)$ is a subodometer. 
    As a subodometer $(X,G)$ is uniformly regularly recurrent and hence $(\hyper(X),G)$ is uniformly regularly recurrent. 
    A similar statement holds for $(\mathcal{M}(X),G)$. 
\end{remark}

\subsection{Intrinsic regular recurrence}
\label{subsec:RegularRecurrence_IntrinsicRegularRecurrence}
For subodometers, any point is regularly recurrent. The $\Gamma$ witnessing the regular recurrence of a point can be chosen as an eigenvalue as we present next. 

\begin{lemma}
\label{lem:regularRecurrenceAndEigensets}
If $(X,G)$ is a subodometer, then for any $\epsilon\in \mathbb{U}_X$ and $x\in X$ there exists $\Gamma\in \Eig_x(X,G)$ with
$\Gamma.x\subseteq B_\epsilon(x)$. 
\end{lemma}
\begin{proof}
    By Proposition \ref{pro:characterizationEquicontinuousStone} there exists an invariant closed equivalence relation $\rho\in \mathbb{U}_X$ with $\rho\subseteq \epsilon$. 
    Consider the factor $\pi\colon X\to X/ \rho=:Y$ and note that $Y$ is finite. 
    From $\Gamma:=G_0(\pi(x))\in \Eig_x(X,G)$ we observe $g.x\in B_\rho(x)\subseteq B_\epsilon(x)$ for all $g\in \Gamma$. 
\end{proof}

This motivates the following definition. 
A minimal action $(X,G)$ is called \emph{intrinsically regularly recurrent} if for any $\epsilon\in \mathbb{U}_X$ there exists $\Gamma\in \Eig(X,G)$ such that for all $x\in X$ we have $\Gamma.x\subseteq B_\epsilon(x)$. 
Clearly any intrinsically regularly recurrent action is uniformly regularly recurrent. That the converse is not necessarily true can be observed from the following. 

\begin{proposition}
\label{pro:characterizationOdometerAsIntrinsicallyRegularlyRecurrent}
    A minimal action $(X,G)$ is an odometer if and only if it is intrinsically regularly recurrent. 
\end{proposition}

\begin{proof}
    Assume that $(X,G)$ is an odometer and consider $\epsilon\in \mathbb{U}_X$. 
    Recall from Proposition \ref{pro:characterizationEquicontinuousStone} that $\mathbb{U}_X$ allows for a base of invariant equivalence relations. 
    Thus, we assume w.l.o.g.\ that $\epsilon$ is an invariant equivalence relation. 
    Since $X$ is compact there exists $F\subseteq X$ finite such that $X=\bigcup_{x\in F}B_\epsilon(x)$. 
    By Lemma \ref{lem:regularRecurrenceAndEigensets} there exist $\Gamma_x\in \Eig(X,G)$ such that for all $x\in F$ we have $\Gamma_x.x\subseteq B_\epsilon(x)$. 
    Since $(X,G)$ is an odometer we know from Proposition \ref{pro:odometerCharacterizationEIG} that $\Eig(X,G)$ is intersection closed and hence $\Gamma:=\bigcap_{x\in F}\Gamma_x\in \Eig(X,G)$. 
    For $x\in X$ and $g\in \Gamma$ there exists $x'\in F$ with $x\in B_\epsilon(x')$, i.e.\ $(x',x)\in \epsilon$.
    From $x'\in F$ we observe $(g.x',x')\in \epsilon$. 
    Furthermore, the invariance of $\epsilon$ yields $(g.x,g.x')\in \epsilon$
    and hence $(g.x,x)\in \epsilon\epsilon\epsilon=\epsilon$.
    Thus, $\Gamma.x\subseteq B_\epsilon(x)$ holds for all $x\in X$. 
    This shows that $(X,G)$ is intrinsically regularly recurrent. 

    For the converse assume that $(X,G)$ is intrinsically regularly recurrent. In order to show that $(X,G)$ is an odometer we use Proposition \ref{pro:odometerCharacterizationEIG} and show that $\Eig(X,G)$ is a scale. 
    For this consider $\Gamma_1,\Gamma_2\in \Eig(X,G)$. 
    For $i\in \{1,2\}$ consider the factor map $\pi_i\colon X\to G/ \Gamma_i=:Y_i$ and $\epsilon_i:=(\pi_i\times \pi_i)^{-1}(\Delta_{Y_i})$. 
    Since $Y_i$ is finite we have $\Delta_{Y_i}\in \mathbb{U}_{Y_i}$ and hence $\epsilon_i\in \mathbb{U}_{X}$. 
    It follows that $\epsilon:=\epsilon_1\cap \epsilon_2\in \mathbb{U}_X$. 
    From the intrinsic regular recurrence of $(X,G)$ we know that there exists $\Gamma\in \Eig(X,G)$ with $\Gamma.x\subseteq B_\epsilon(x)$ for all $x\in X$. 
    To show that $\Gamma\subseteq \Gamma_1\cap \Gamma_2$ let $g\in \Gamma$ and $i\in \{1,2\}$.
    Choose $x\in \pi_i^{-1}(\Gamma_i)$. 
    From $g\in \Gamma$ we know $g.x\in B_\epsilon(x)$ and hence $(g.x,x)\in \epsilon\subseteq \epsilon_i$. 
    It follows that $g.\pi_i(x)=\pi_i(g.x)=\pi_i(x)$. 
    We thus have 
    $g\in G_0(\pi_i(x))=G_0(\Gamma_i)=\Gamma_i$. 
\end{proof}

\section{Hyperspaces of subodometers}
\label{sec:hyperspacesOfSubodometers}
Whenever $(X,G)$ is a minimal action, the minimal subactions of $(\hyper(X),G)$ are called ($\hyper$-)\emph{quasifactors}. 
For a subodometer $(X,G)$ we have shown that $(\hyper(X),G)$ decomposes into a disjoint union of quasifactors, all being subodometers. 
It is natural to ask how these are related to the subodometer $(X,G)$. 
A first insight is provided by the following well-known statement (see, for instance, the discussion in the introduction of \cite{glasner1995quasifactors}). We include the short proof\footnote{
    Note that similarly one can show that for a minimal distal action any factor is a quasifactor. 
}
for the convenience of the reader.

\begin{proposition}
\label{pro:hyperspacefactorsAsQuasifactors}
    Let $(X,G)$ be a subodometer. 
    Any factor of $(X,G)$ is conjugated to some quasifactor of $(X,G)$. 
\end{proposition}
\begin{proof}
    Let $\pi\colon X\to Y$ be a factor map. 
    Since $(X,G)$ is minimal and equicontinuous we observe that $\pi$ is an open map \cite[Theorem~7.3]{auslander1988minimal}. 
    In particular, $\phi\colon Y\to \hyper(X)$ with $y\mapsto \pi^{-1}(y)$ is continuous \cite[Appendix A.7]{deVries1993elements}.
    Since $\phi$ is clearly a bijection onto its image $\phi(Y)$, we observe it to be a homeomorphism. 
    Since $\pi$ is a factor map we observe $\phi$ to be a conjugacy between $(Y,G)$ and the quasifactor $(\phi(Y),G)$. 
\end{proof}

Note that any odometer $(X,G)$ is regular. 
It follows from \cite[Corollary~11.21]{auslander1988minimal} that any quasifactor of $(X,G)$ is conjugated to a factor of $(X,G)$. 
Hence, we have the following. 

\begin{proposition}
\label{pro:hyperspaceOdometersQFareExactlyF}
    Whenever $(X,G)$ is an odometer, then (up to conjugacy) the quasifactors of $(X,G)$ are exactly the factors of $(X,G)$. 
\end{proposition}

For subodometers the situation is more complicated as we will see in Theorem \ref{the:hyperspaceCharacterizationOdometerViaQuasifactors} below. 
However, we have the following, which also follows from combining the results of \cite{glasner1975compressibility,cortez2026minimal}. 

\begin{corollary}
\label{cor:hyperspaceQuasifactorsAreFactorsOfEnvelopingOdometer}
    Let $(X,G)$ be a subodometer and denote $(\envelopingOdometer{X},G)$ for its enveloping odometer. Any quasifactor of $(X,G)$ is a factor of $(\envelopingOdometer{X},G)$.  
\end{corollary}
\begin{proof}
    Consider a factor map $\pi\colon \envelopingOdometer{X}\to X$. 
    It is straightforward to verify that $\pi_\hyper \colon \hyper(\envelopingOdometer{X})\to \hyper(X)$ given by $A\mapsto \pi(A)$ establishes a factor map. 
    Thus, any quasifactor of $(X,G)$ is a factor of a quasifactor of $(\envelopingOdometer{X},G)$ and the statement follows from Proposition \ref{pro:hyperspaceOdometersQFareExactlyF}. 
\end{proof}

It is natural to ask which subodometers allow for quasifactors that are not factors. 
As we will see in Subsection \ref{subsec:hyperspacesOfSubodometers_QuasifactorsOfSmallOdometers} below, there exist finite subodometers, that are not odometers, that have only factors as quasifactors.
However, for infinite subodometers this phenomenon does not occur. 

\begin{theorem}
\label{the:hyperspaceCharacterizationOdometerViaQuasifactors}
    Let $(X,G)$ be an infinite subodometer. 
    The following statements are equivalent.
    \begin{enumerate}[label=(\roman*)]
        \item \label{enu:qfXodometer}
            $(X,G)$ is an odometer.
        \item \label{enu:qfXregular}
            $(X,G)$ is regular. 
        \item \label{enu:qfallQFareF}
            All quasifactors of $(X,G)$ are factors. 
        \item \label{enu:qfallfinQFareF}
            All finite quasifactors of $(X,G)$ are factors. 
    \end{enumerate}
\end{theorem}

For the preparation of the proof of Theorem \ref{the:hyperspaceCharacterizationOdometerViaQuasifactors} we introduce the following notion. 

\subsection{Quasifactors via settled subgroups}
\label{subsec:hyperspacesOfSubodometers_quasifactorsViaSettledSubgroups}

\begin{definition}
    Let $(X,G)$ be a minimal action. 
    We call $\Gamma\in \Eig(X,G)$ \emph{settled} (w.r.t.\ $(X,G)$) if for all $g\in G$ with $\Gamma^g \cap \Gamma\in \Eig(X,G)$ we have $\Gamma^g=\Gamma$. 
\end{definition}

The notion of settledness can be used as follows for the construction of quasifactors that are not factors. 
Recall that for $\Lambda\leq \Gamma\leq G$ we denote $\pi^{\Lambda}_{\Gamma}\colon G/\Lambda\to G/\Gamma$ for the factor map with $\Lambda\mapsto \Gamma$. 

\begin{lemma}
\label{lem:quasiFactorConstruction}
    Let $(X,G)$ be a minimal action, $g\in G$ and $\Gamma,\Lambda\in \Eig(X,G)$ such that $\Lambda\subseteq \Gamma$. 
    There exist factor maps $\pi_\Lambda\colon X\to G/\Lambda$ and 
    $\pi_\Gamma\colon X\to G/\Gamma$ with $\pi_\Gamma=\pi^\Gamma_\Lambda\circ \pi_\Lambda$. 
    Consider 
    \[A := \pi_\Gamma^{-1}(\Gamma)\cup \pi_\Lambda^{-1}(g\Lambda)\in \hyper(X).\]
\begin{itemize}
    \item[(i)] If $g\notin \Gamma$ and $[\Gamma\colon \Lambda]\geq 3$, then $G_0(A)=\Gamma\cap \Lambda^g$. 
    \item[(ii)] If $\Gamma$ is settled, $\Gamma^g\neq \Gamma$ and $[\Gamma\colon \Lambda]\geq 3$, then 
    $(\overline{G.A},G)$ is a finite quasifactor of $(X,G)$ that is not a factor.
\end{itemize}
\end{lemma}
\begin{proof}
    Since $\Lambda\in \Eig(X,G)$ there exists a factor map $\pi_\Lambda\colon X\to G/\Lambda$ and $\pi_\Gamma:=\pi^\Lambda_\Gamma\circ \pi_\Lambda$ yields a factor map $X\to G/\Gamma$. 

(i):
    Clearly, we have $G_0(A)\supseteq \Gamma\cap \Lambda^g$. 
    For the converse consider $h\in G_0(A)$. 
    Let $F\subseteq \Gamma$ be a fundamental domain of $\Lambda\leq \Gamma$ and denote
    $F':=F\cup \{g\}$.
    Note that
    \begin{align*}
        A
        =\pi^{-1}(\{g'\Lambda;\, g'\in F\}) \cup \pi^{-1}(g\Lambda)
        =\pi^{-1}\left(\{g'\Lambda;\, g'\in F'\}\right). 
    \end{align*}
    Thus, $h$ establishes a permutation $\sigma\colon F'\to F'$ with 
    $hg'\Lambda=\sigma(g')\Lambda$ for $g'\in F'$. 

    We next show that $h\in \Gamma$. 
    For this recall that $|F|=[\Gamma\colon \Lambda]\geq 3$.    
    Thus, there exist $g_1,g_2\in F$ with $\sigma(g_1)=g_2$. 
    We have $hg_1\in hg_1\Lambda = g_2\Lambda$ and hence 
    \[h\in g_2\Lambda g_1^{-1}\subseteq F\Lambda F\subseteq\Gamma\Gamma\Gamma=\Gamma.\] 

    It remains to show that $h\in \Lambda^g$. 
    For this recall that $g\in F'$. 
    If $\sigma(g)\neq g$ we have $\sigma(g)\in F$ and hence 
    \[g
    =h^{-1}hg\in h^{-1}hg\Lambda
    =h^{-1}\sigma(g)\Lambda
    \subseteq \Gamma F \Lambda
    \subseteq \Gamma\Gamma\Gamma=\Gamma,\]
    a contradiction. 
    This shows $\sigma(g)=g$ and it follows that
    \[h=hgg^{-1}\in hg\Lambda g^{-1}=\sigma(g)\Lambda g^{-1}=g\Lambda g^{-1}=\Lambda^g.\]

(ii):
    Note that $\Gamma^g\neq \Gamma$ implies $g\notin \Gamma$. 
    We thus observe from (i) that $G_0(A)=\Gamma\cap \Lambda^g$ is a finite index subgroup of $G$. 
    It follows that $(\overline{G.A},G)$ is finite. 

    If $(\overline{G.A},G)$ is a factor of $(X,G)$, then 
    $\Gamma\cap \Lambda^g=G_0(A)\in \Eig(X,G)$. 
    Since $\Eig(X,G)$ is upward closed we observe from 
    $\Gamma\cap \Lambda^g\subseteq \Gamma\cap \Gamma^g$ that 
    $\Gamma\cap \Gamma^g\in \Eig(X,G)$.
    From $\Gamma$ being settled it follows that $\Gamma=\Gamma^g$, a contradiction. 
    This shows that $(\overline{G.A},G)$ is not a factor of $(X,G)$. 
\end{proof}

\subsection{Odometers and settled subgroups}
\label{subsec:hyperspacesOfSubodometers_OdometersAndSettledSubgroups}
In consideration of the previous lemma we are interested in the existence of settled eigenvalues that are not normal. We next show that a subodometer is an odometer if and only if all settled eigenvalues are normal. 
To establish this we will use the following notion. 

\begin{definition}
    For $\Gamma\in \subfin(G)$ and $I\subseteq G$ we denote \[\Gamma_I:=\bigcap_{g\in I}\Gamma^g.\]     
\end{definition}

\begin{lemma}
\label{lem:conjugationInsights}
    Let $\Gamma\in \subfin(G)$, $g\in G$ and $I\subseteq G$. 
    \begin{itemize}
        \item[(i)] $(\Gamma_I)^g=\Gamma_{gI}$. 
        \item[(ii)] If $\Gamma\subseteq \Gamma^g$, then $\Gamma=\Gamma^g$. 
        \item[(iii)] For a fundamental domain $F$ of $\Gamma$ there exists $J\subseteq F$ with $\Gamma_I=\Gamma_J$. 
    \end{itemize}
\end{lemma}
\begin{proof}
(i):
    $(\Gamma_I)^g=g\left(\bigcap_{h\in I} h\Gamma h^{-1}\right)g^{-1}=\bigcap_{h\in I} gh\Gamma (gh)^{-1}=\Gamma_{gI}.$

(ii): 
    Note that $[G\colon \Gamma]=[G\colon \Gamma^g]$.
    The statement thus follows from Remark \ref{rem:fundamentalDomainsComposition}. 

(iii): 
    For $g\in I\subseteq G$ there exist $f_g\in F$ and $h_g\in \Gamma$ with $g=f_gh_g$. We have $g\Gamma=f_gh_g\Gamma=f_g\Gamma$ and hence
    \[\Gamma^g=g\Gamma g^{-1}=(g\Gamma)(g\Gamma)^{-1}=(f_g\Gamma)(f_g\Gamma)^{-1}=f_g\Gamma f_g^{-1}=\Gamma^{f_g}.\]
    Thus, $J:=\{f_g;\, g\in I\}$ satisfies $\Gamma_J=\Gamma_I$. 
\end{proof}

\begin{lemma}
\label{lem:creatingEsettledSets}
    Let $(X,G)$ be a minimal action. 
    For every $\Gamma\in \Eig(X,G)$ 
    there exists a finite subset $I\subseteq G$, such that $\Gamma_I$ is settled. 
\end{lemma}
\begin{proof}
    Consider $\Gamma\in \Eig(X,G)$. 
    Let $F$ be a fundamental domain of $\Gamma$ with $e_G\in F$. 
    Denote by $\mathcal{I}$ the family of all $I\subseteq F$ with $\Gamma_I\in \Eig(X,G)$. 
    Since $F$ is finite we observe $\mathcal{I}$ to be finite. 
    For $I=\{e_G\}$ we observe $\Gamma_{I}=\Gamma\in \Eig(X,G)$. 
    This shows that $\mathcal{I}$ is a nonempty finite set. 
    Thus, there exists a maximal element $I$ w.r.t.\ set inclusion. 

    To show that $\Gamma_I$ is settled, consider $g\in G$ with 
    $(\Gamma_I)^g\cap \Gamma_I\in \Eig(X,G)$. 
    Since $F$ is a fundamental domain of $\Gamma\leq G$ we know from Lemma \ref{lem:conjugationInsights} that there exists a subset $J\subseteq F$ with $\Gamma_{J}=\Gamma_{gI}$. 
    Denote $\hat{I}:=I\cup J$ and note that $\hat{I}\subseteq F$. 
    Clearly, we have
    $I\subseteq \hat{I}$ and $\Gamma_J
    =\Gamma_{gI}
    =(\Gamma_I)^g.$
    Thus, 
    $\Gamma_{\hat{I}}=\Gamma_{J}\cap \Gamma_{I}=(\Gamma_{I})^g\cap \Gamma_{I}\in \Eig(X,G)$, i.e.\ $\hat{I}\in \mathcal{I}$.    
    Since $I$ is a maximal element of $\mathcal{I}$ we observe $I=\hat{I}$. 
    This shows $J\subseteq I$ and hence
    $(\Gamma_{I})^g=\Gamma_{gI}=\Gamma_{J}\supseteq \Gamma_{I}$. 
    It follows from Lemma \ref{lem:conjugationInsights} that $(\Gamma_{I})^g=\Gamma_{I}$.  
\end{proof}

Note that any normal $\Gamma\in \Eig(X,G)$ is settled.
We next present that the converse characterizes odometers. 

\begin{proposition}
\label{pro:filteredEigensetsViaEsettled}
   A subodometer $(X,G)$ is an odometer if and only if all settled $\Gamma\in \Eig(X,G)$ are normal. 
\end{proposition}
\begin{proof}
    Assume that $(X,G)$ is an odometer and consider a settled $\Gamma\in \Eig(X,G)$. 
    To show that $\Gamma$ is normal let $g\in G$. 
    From $\Eig(X,G)$ being conjugation invariant we observe $\Gamma^g\in \Eig(X,G)$. 
    Since $(X,G)$ is an odometer we know from Proposition \ref{pro:odometerCharacterizationEIG} that $\Eig(X,G)$ is intersection closed and hence $\Gamma\cap \Gamma^g\in \Eig(X,G)$. 
    By the settledness of $\Gamma$, we have $\Gamma=\Gamma^g$.
    This shows $\Gamma$ to be normal. 

    For the converse assume that all settled $\Gamma\in \Eig(X,G)$ are normal. In order to show that $(X,G)$ is an odometer we use Proposition \ref{pro:odometerCharacterizationEIG} and show that $\Eig(X,G)$ is core-stable. 
    For this consider $\Lambda\in \Eig(X,G)$. 
    From Lemma \ref{lem:creatingEsettledSets} we observe the existence of $I\subseteq G$ finite, such that $\Lambda_I\in \Eig(X,G)$ is settled and hence normal. In particular, we have $(\Lambda_I)^g=\Lambda_I$ for all $g\in G$ and hence
    \[\Lambda_G=(\Lambda_I)_G=\Lambda_I\in \Eig(X,G).\]
    This shows that $\Eig(X,G)$ is core-stable. 
\end{proof}

\subsection{Proof of Theorem \ref{the:hyperspaceCharacterizationOdometerViaQuasifactors}}
\label{subsec:hyperspacesOfSubodometers_ProofOfTheoremCOVQ}
In order to impose the condition $[\Gamma\colon \Lambda]\geq 3$ from Lemma \ref{lem:quasiFactorConstruction} we will use the following. 

\begin{lemma}
\label{lem:indicesInInfiniteEigensets}
   Let $(X,G)$ be an infinite subodometer. 
   For $\Gamma\in \Eig(X,G)$ there exists $\Lambda\in \Eig(X,G)$ with $\Lambda\subsetneq \Gamma$. 
\end{lemma}
\begin{proof}
    From $\Gamma\in \Eig(X,G)$ we observe that there exists a factor map $\pi\colon X\to G/\Gamma$. 
    Choose $x\in \pi^{-1}(\Gamma)$. 
    Clearly, we have $\Gamma\in \Eig_x(X,G)$. 

    To achieve the contraposition assume that there exists no $\Lambda\in \Eig(X,G)$ with $\Lambda\subsetneq \Gamma$. 
    Since $\Eig_x(X,G)$ is a filter for $\Lambda\in \Eig_x(X,G)$ 
    we have $\Gamma\cap\Lambda\in \Eig(X,G)$. 
    It follows that $\Gamma=\Gamma\cap \Lambda$ and hence $\Gamma\subseteq \Lambda$. 
    This establishes $\Gamma$ as a minimal element of $\Eig_x(X,G)$ w.r.t.\ set inclusion. 
    Recall that $\Eig_x(X,G)$ is a scale that generates $(X,G)$.
    Thus, $(X,G)$ is conjugated to the finite subodometer $G/\Gamma$, a contradiction. 
\end{proof}

\begin{proof}[Proof of Theorem \ref{the:hyperspaceCharacterizationOdometerViaQuasifactors}:]
    Any odometer is regular. 
    Furthermore, it follows from \cite[Corollary~11.21]{auslander1988minimal} that any quasifactor of a regular subodometer is also a factor. 
    It thus remains to show that any infinite subodometer that is not an odometer allows for a finite quasifactor that is not a factor.

    Consider a subodometer $(X,G)$ that is not an odometer. 
    Since $(X,G)$ is not an odometer we observe from Proposition \ref{pro:filteredEigensetsViaEsettled} that there exists $\Gamma\in \Eig(X,G)$ that is settled and not normal. 
    In particular, there exists $g\in G$ with $\Gamma^g\neq \Gamma$. 
    Since $(X,G)$ is infinite an inductive application of Lemma \ref{lem:indicesInInfiniteEigensets} allows to choose $\Lambda\in \Eig(X,G)$ with $\Lambda\subseteq \Gamma$ and $[\Gamma\colon \Lambda]\geq 3$. 
    It follows from Lemma \ref{lem:quasiFactorConstruction} that there exists a finite quasifactor of $(X,G)$ which is not a factor. 
\end{proof}

\subsection{Quasifactors of small subodometers}
\label{subsec:hyperspacesOfSubodometers_QuasifactorsOfSmallOdometers}
We next present that for a finite subodometer that is not an odometer we can have that all quasifactors are factors. 

\begin{proposition}
\label{pro:verySmallPhaseSpaces}
    Let $(X,G)$ be a minimal action with $|X|\leq 3$. Then any quasifactor of $(X,G)$ is a factor.         
\end{proposition}
\begin{proof}
    Denote $\mathcal{F}_k(X)$ for the set of all finite subsets of $X$ of cardinality $k$. We observe that $(\mathcal{F}_1(X),G)$ is conjugated to $(X,G)$ and that $(\mathcal{F}_{|X|}(X),G)$ is conjugated to the one-point system. 
    For $|X|\leq 2$ we have found all quasifactors to be factors. For $|X|=3$ we observe that the remaining part of $\hyper(X)$ yields the quasifactor $(\mathcal{F}_2(X),G)$, which is conjugated to $(X,G)$ via $X\ni x\mapsto X\setminus \{x\}\in \hyper(X)$. 
\end{proof}

Note that any finite index subgroup of index $2$ is normal. Thus any minimal action on a phase space $X$ with $|X|\leq 2$ is an odometer. 
The following examples demonstrate that there exist subodometers ($X,G$), not being odometers with $|X|=3$. 

\begin{example}
    Let $H$ be a group and consider $G:=H\times S_3$ and $\Gamma:=H\times \grouphull{(12)}$. 
    The subodometer $(G/\Gamma,G)$ satisfies $|G/\Gamma|=3$. 
    Since $\Gamma$ is not normal it is not an odometer. 
\end{example}

\begin{example}
    Let $F_2=\grouphull{\{g_1,g_2\}}$ be the free group on two generators. 
    Consider the group homomorphism $\phi\colon F_2\to S_3$ with $\phi(g_1)=(123)$ and $\phi(g_2)=(12)$. 
    Since $\grouphull{(12)}$ is not normal and $\phi$ is surjective, $\Gamma:=\phi^{-1}(\grouphull{(12)})$ is a non-normal finite index subgroup of $F_2$ and hence the subodometer $(F_2/\Gamma,F_2)$ is not an odometer.
    Nevertheless, it satisfies $|F_2/\Gamma|=|S_3/\grouphull{(12)}|=3$. 
\end{example}

\section{Induced dynamics on the regular Borel probability measures}
\label{sec:inducedDynamicsOnTheRegularBorelProbabilityMeasures}
For a minimal action $(X,G)$ we call a minimal component of $(\mathcal{M}(X),G)$ a \linebreak \emph{$\mathcal{M}$-quasifactor}. 
If $(X,G)$ is a subodometer, then we know from Corollary \ref{cor:InducedDynamicsDecompose} that $(\mathcal{M}(X),G)$ decomposes into $\mathcal{M}$-quasifactors, all being subodometers. 
We will next show that the $\mathcal{M}$-quasifactors exhibit similar properties to the $\hyper$-quasifactors.

\subsection{Factors as $\mathcal{M}$-quasifactors}
\label{subsec:inducedDynamicsOnTheRegularBorelProbabilityMeasures_FactorsAsMQuasifactors}
Next, we show that any factor of a subodometer is a $\mathcal{M}$-quasifactor, i.e.\ the analog of Proposition \ref{pro:hyperspacefactorsAsQuasifactors}. 
For this we need the following. Recall that we denote $G_\epsilon(x):=\{g\in G;\, (g.x,x)\in \epsilon\}$ 
for an action $(X,G)$, $x\in X$ and $\epsilon\in \mathbb{U}_X$. 

\begin{lemma}
\label{lem:measureConstructionAndControl}
    Let $(X,G)$ be a subodometer and $S\subseteq \Eig(X,G)$ be a scale. 
    There exist nets $(\mu_\Gamma)_{\Gamma\in S}$ in $\mathcal{M}(X)$ and $(\eta_\Gamma)_{\Gamma\in S}$ in $\mathbb{U}_{\mathcal{M}(X)}$, such that for 
    $\Gamma,\Lambda\in S$ with $\Lambda\subseteq \Gamma$ we have 
    $\Gamma=G_0(\mu_\Gamma)=G_{\eta_\Gamma}(\mu_\Lambda)$. 
\end{lemma}
\begin{proof}
    Denote $\theta$ for the unique invariant regular Borel probability measure on $X$ \cite[Chapter~7]{auslander1988minimal}. 
    Denote $(Y,G)$ for the $S$-subodometer. Since $S\subseteq \Eig(X,G)$ we have $\Eig(Y,G)\subseteq \Eig(X,G)$ and hence there exists a factor map $\pi\colon X\to Y$. 
    Denote $y_0:=(\Gamma)_{\Gamma\in S}$ and choose $x_0\in \pi^{-1}(y_0)$. 
    For $\Gamma\in S$ there exists a unique factor map $\phi_\Gamma\colon Y\to G/\Gamma$ with $\phi_\Gamma(y_0)=\Gamma$. 
    Denote $\pi_\Gamma:=\phi_\Gamma\circ \pi\colon X\to/\Gamma$ and note that $\pi_\Gamma(x_0)=\Gamma$. 
    Furthermore, we have $R(\pi_\Lambda)\subseteq R(\pi_\Gamma)$ for $\Gamma,\Lambda\in S$ with $\Lambda\subseteq \Gamma$. 
    
    For $\Gamma\in S$ we consider the finite equivalence relation $R(\pi_\Gamma)$. Imposing $x_0\in A_1^\Gamma$ we denote $\{A_1^\Gamma,\dots, A_n^\Gamma\}$ for the induced equivalence classes.  
    Let $f_k^\Gamma:=\chi_{A_k^\Gamma}$ be the characteristic function of $A_k^\Gamma$. 
    With $\mathfrak{f}^\Gamma:=\{f_1^\Gamma,\dots, f_n^\Gamma\}$ and $\epsilon^\Gamma:=(2|\Gamma|)^{-1}$
    we denote $\eta_\Gamma:=\eta_{(\mathfrak{f}^\Gamma,\epsilon^\Gamma)}$. 
    Note that $\theta(A_1^\Gamma)=1/|\Gamma|>0$. 
    Thus, 
    \[\mu_\Gamma:=\theta((\cdot)\cap A_1^\Gamma)/\theta(A_1^\Gamma)\in \mathcal{M}(X).\]
    It is straightforward to see that $G_0(\mu_\Gamma)=\Gamma$. 

    Consider $\Gamma,\Lambda\in S$ with $\Lambda\subseteq \Gamma$.
    In order to show that $G_{\eta_\Gamma}(\mu_\Lambda)=\Gamma$ we first show that $\Gamma\subseteq G_{\eta_\Gamma}(\mu_\Lambda)$. 
    Let $g\in \Gamma=G_0(\mu_\Gamma)$ and observe $g.\mu_\Gamma=\mu_\Gamma$. 
    Furthermore, for $f\in \mathfrak{f}^\Gamma$ we have 
    $\mu_\Lambda(f)=\mu_\Gamma(f)$. Since $g^*f\in \mathfrak{f}^\Gamma$ we observe
    \begin{align*}
        |\mu_\Lambda(f)- g.\mu_\Lambda(f)|
        = |\mu_\Lambda(f)- \mu_\Lambda(g^*f)|
        = |\mu_\Gamma(f)- g.\mu_\Gamma(f)|
        = 0 \leq \epsilon^\Gamma. 
    \end{align*}
    This shows $(\mu_\Lambda,g.\mu_\Lambda)\in \eta_\Gamma$ and hence $g\in G_{\eta_\Gamma}(\mu_\Lambda)$. 

    For the converse consider $g\in G_{\eta_\Gamma}(\mu_\Lambda)$. 
    From $A_1^\Lambda\subseteq A_1^\Gamma$ we observe $\mu_\Lambda(f_1^\Gamma)=1$. 
    Furthermore, from $(\mu_\Lambda, g.\mu_\Lambda)\in \eta_\Gamma$ we know
    \[|\mu_\Lambda (f_1^\Gamma)-g.\mu_\Lambda(f_1^\Gamma)|\leq\epsilon^\Gamma<1.\]
    Since $g^*f_1^\Gamma\in \mathfrak{f}^\Gamma$ we have $g.\mu_\Lambda(f_1^\Gamma)=\mu_\Lambda(g^*f_1^\Gamma)\in \{0,1\}$. Combining these observations we get
    $\theta(A_1^\Gamma\cap g.A_1^\Lambda)/\theta(A_1^\Lambda)=g.\mu_\Lambda(f_1^\Gamma)=1$ and hence $A_1^\Gamma\cap g.A_1^\Lambda\neq \emptyset$. 
    It follows from $R(\pi_\Lambda)\subseteq R(\pi_\Gamma)$ that 
    $g.x_0\in g.A_1^\Lambda \subseteq A_1^\Gamma$. 
    We thus observe 
    \[g\Gamma 
    =g.\pi_\Gamma(x_0)
    =\pi_\Gamma(g.x_0)
    \subseteq \pi_\Gamma(A_1^\Gamma)
    =\Gamma,
    \]
    which implies $g\in \Gamma$. 
    This shows $G_{\eta_\Gamma}(\mu_\Lambda)\subseteq\Gamma$. 
\end{proof}

\begin{lemma}\cite[Lemma~4.10]{cortez2026minimal}
\label{lem:factorMapControl}
    Let $(X,G)$ and $(Y,G)$ be equicontinuous minimal actions and consider $x\in X$ and $y\in Y$. 
    There exists a factor map $\pi\colon X\to Y$ with $\pi(x)=y$ if and only if for any $\epsilon\in \mathbb{U}_Y$ there exists $\delta\in \mathbb{U}_X$ such that 
    $G_\delta(x)\subseteq G_\epsilon(y)$. 
\end{lemma}

\begin{theorem}
\label{the:factorsAreMQuasifactors}
    Let $(X,G)$ be a subodometer. 
    Any factor of $(X,G)$ is conjugated to a $\mathcal{M}$-quasifactor of $(X,G)$.  
\end{theorem}
\begin{proof}
    Let $(Y,G)$ be a factor of $(X,G)$ and note that $(Y,G)$ is a subodometer. 
    Thus, it allows for a generating scale $S$ with $S\subseteq \Eig(X,G)$. 
    Consider the nets $(\mu_\Gamma)_{\Gamma\in S}$ in $\mathcal{M}(X)$ and $(\eta_\Gamma)_{\Gamma\in S}$ in $\mathbb{U}_{\mathcal{M}(X)}$ as constructed in Lemma \ref{lem:measureConstructionAndControl}. 
    Since $\mathcal{M}(X)$ is compact $(\mu_\Gamma)_{\Gamma\in S}$ admits cluster points. By possibly restricting to another scale for $(Y,G)$ we assume w.l.o.g.\ that $\mu_\Gamma \to \mu$ in $\mathcal{M}(X)$. Denote $Z:=\overline{G.\mu}$. 
    In order to establish that $(Y,G)$ and $(Z,G)$ are conjugated we next show that $S$ and $\Eig_\mu(Z,G)$ are equivalent scales. 
    
    We first show that $S\subseteq \Eig_\mu(Z,G)$. 
    Consider $\Gamma\in S$. Since $(\mathcal{M}(X),G)$ is equicontinuous there exists an invariant and symmetric $\eta\in \mathbb{U}_{\mathcal{M}(X)}$ such that $\eta\eta\eta\subseteq \eta_\Gamma$. 
    Consider $g\in G_\eta(\mu)$ and choose $\Lambda\in S$ with $(\mu_\Lambda,\mu)\in \eta$. 
    Note that $(\mu,g.\mu)\in \eta$. 
    Furthermore, since $\eta$ is invariant and symmetric we know $(g.\mu,g.\mu_\Lambda)\in \eta$. 
    We thus have $(\mu_\Lambda,g.\mu_\Lambda)\in \eta\eta\eta\subseteq \eta_\Gamma$ and hence $g\in G_{\eta_\Gamma}(\mu_\Lambda)=\Gamma$. 
    This shows $G_\eta(\mu)\subseteq \Gamma=G_0(\Gamma)$ and Lemma \ref{lem:factorMapControl} yields the existence of a factor map $Z\to G/\Gamma$ with $\mu\mapsto \Gamma$, i.e.\ $\Gamma\in \Eig_\mu(Z,G)$. 

    To show that $S$ dominates $\Eig_\mu(Z,G)$, consider $\Gamma\in \Eig_\mu(Z,G)$. 
    Let $\pi\colon Z\to G/\Gamma$ be the factor map with $\pi(\mu)=\Gamma$ and denote 
    $\eta:=R(\pi)\cup (\mathcal{M}(X)\setminus Z)^2.$
    Since $Z$ is a closed subset of $\mathcal{M}(X)$ and $G/\Gamma$ is finite we observe that $\eta$ is open and hence $\eta\in \mathbb{U}_{\mathcal{M}(X)}$. 
    Furthermore, we have $G_\eta(\mu)=G_0(\Gamma)=\Gamma$. 
    Since $(\mathcal{M}(X),G)$ is equicontinuous we find an invariant and symmetric $\delta\in \mathbb{U}_{\mathcal{M}(X)}$ such that $\delta\delta\subseteq \eta$. 
    Let $\Lambda\in S$ be such that $(\mu,\mu_\Lambda)\in \delta$. 
    For $g\in \Lambda=G_0(\mu_\Lambda)$ we know that $g.\mu_\Lambda=\mu_\Lambda$ and hence the invariance of $\delta$ yields
    $(g.\mu, \mu_\Lambda)=(g.\mu,g.\mu_\Lambda)\in \delta$. 
    Thus, we have $(\mu,g.\mu)\in \delta\delta\subseteq \eta$ and hence 
    $g\in G_\eta(\mu)=\Gamma$. 
    This shows $\Lambda\subseteq \Gamma$ and establishes that $S$ dominates $\Eig_\mu(Z,G)$. 
\end{proof}

\subsection{$\mathcal{M}$-quasifactors of odometers}
\label{subsec:inducedDynamicsOnTheRegularBorelProbabilityMeasures_MQuasifactorsOfOdometers}
Next, we show the analog of Proposition \ref{pro:hyperspaceOdometersQFareExactlyF}, i.e.\ that for odometers (up to conjugacy) the $\mathcal{M}$-quasifactors are exactly the factors.
For this we need the following. 

\begin{lemma}
\label{lem:odometerUniformRegularRecurrenceEigenvalueChoice}
    Let $(X,G)$ be an odometer. 
    \begin{itemize}
        \item[(i)] 
        For $f\in C(X)$ and $\epsilon>0$ there exists a finite index subgroup $\Gamma\in \Eig(X,G)$ such that $\sup_{g\in \Gamma} \|f-g^*f\|_\infty\leq \epsilon.$   
        \item[(ii)] 
        For any $\eta\in \mathbb{U}_{\mathcal{M}(X)}$ there exists $\Gamma\in \Eig(X,G)$ such that $\Gamma.\mu\subseteq B_\eta(\mu)$ holds for all $\mu\in  \mathcal{M}(X)$.
    \end{itemize}
\end{lemma}
\begin{proof}
(i): 
    Recall from Proposition \ref{pro:characterizationOdometerAsIntrinsicallyRegularlyRecurrent} that odometers are intrinsically regularly recurrent. 
    Thus, in the proof of Lemma \ref{lem:regularRecurrenceInCX} we can choose $\Gamma\in \Eig(X,G)$ to observe the statement. 

(ii): 
    Reconsidering the proof of Theorem \ref{the:HMinheritanceUniformRegularRecurrence} we can invoke (i) in order to choose $\Gamma_f\in \Eig(X,G)$ for $f\in \mathfrak{f}$.
    Recall from Proposition \ref{pro:odometerCharacterizationEIG} that for odometers $\Eig(X,G)$ is intersection closed. 
    Thus, we observe that $\Gamma=\bigcap_{f\in \mathfrak{f}}\Gamma_f\in \Eig(X,G)$. 
\end{proof}

\begin{theorem}
\label{the:forOdometersComponentsMXfactor}
    Whenever $(X,G)$ is an odometer, then (up to conjugacy) the $\mathcal{M}$-quasifactors of $(X,G)$ are exactly the factors of $(X,G)$. 
\end{theorem}
\begin{proof}
    It remains to show that each $\mathcal{M}$-quasifactor is a factor of $(X,G)$. 
    For this consider $\mu\in \mathcal{M}(X)$. 
    Denote $Y:=\overline{G.\mu}$ for the respective minimal component and note that $(Y,G)$ is a subodometer. 
    Next, we show that $\Eig_\mu(Y,G)\subseteq \Eig(X,G)$. 
    We then have $\Eig(Y,G)=\eigenhull{\Eig_\mu(Y,G)}\subseteq \Eig(X,G)$ and Proposition \ref{pro:EIGfactorsAndConjugacyInvariant} yields that $(Y,G)$ is a factor of $(X,G)$. 
    
    Consider $\Gamma\in \Eig_\mu(Y,G)$ and the factor map $\pi\colon Y\to G/\Gamma=:Z$ with $\pi(\mu)=\Gamma$. 
    Denote
    $\eta:=R(\pi)\cup (\mathcal{M}(X)\setminus Z)^2$
    and note that $\eta\in \mathbb{U}_{\mathcal{M}(X)}$. 
    Lemma \ref{lem:odometerUniformRegularRecurrenceEigenvalueChoice}
    allows us to find $\Lambda\in \Eig(X,G)$ such that $\Lambda.\nu \subseteq B_\eta(\nu)$ holds for all $\nu\in \mathcal{M}(X)$. 
    It follows that $\Lambda\subseteq G_\eta(\mu)=G_0(\pi(\mu))=\Gamma$. 
    Since $\Eig(X,G)$ is upward closed and contains $\Lambda$ we conclude $\Gamma\in \Eig(X,G)$. 
\end{proof}

With a similar argument as for Corollary \ref{cor:hyperspaceQuasifactorsAreFactorsOfEnvelopingOdometer} we observe the following. 

\begin{corollary}
\label{cor:measureQuasifactorsAreFactorsOfEnvelopingOdometer}
    Let $(X,G)$ be a subodometer and denote by $(\envelopingOdometer{X},G)$ its enveloping odometer. Any $\mathcal{M}$-quasifactor of $(X,G)$ is a factor of $(\envelopingOdometer{X},G)$.  
\end{corollary}

\subsection{$\mathcal{M}$-quasifactors of subodometers}
\label{subsec:inducedDynamicsOnTheRegularBorelProbabilityMeasures_MQuasifactorsOfSubodometers}
Next, we establish the analog of Theorem~\ref{the:hyperspaceCharacterizationOdometerViaQuasifactors} and show that a subodometer is an odometer if and only if all (finite) $\mathcal{M}$-quasifactors are factors. 
For this we use our results on $\hyper$-quasifactors and the following. 

\begin{lemma}
\label{lem:HXMXtranslation}
    Let $(X,G)$ be a minimal equicontinuous action.         
    All finite \linebreak $\hyper$-quasifactors appear as $\mathcal{M}$-quasifactors up to conjugacy. 
\end{lemma}
\begin{proof}
    Let $Y\subseteq \hyper(X)$ be a finite $\hyper$-quasifactor and $A\in Y$.  
    Note that $\Gamma:=G_0(A)\in \subfin(G)$. 
    Since $(X,G)$ is minimal, it is pointwise almost periodic and hence $(X,\Gamma_G)$ decomposes into minimal components. 
    From $\Gamma_G\subseteq \Gamma$ we observe $\overline{\Gamma_G.x}\subseteq A$ for all $x\in X$ and hence $A$ is a finite union of minimal components of $(X,\Gamma_G)$. 
    Furthermore, the normality of $\Gamma_G$ yields that $G$ permutes the minimal components of $(X,\Gamma_G)$.
    It follows that each $A'\in Y$ is a finite union of minimal components of $(X,\Gamma_G)$. 

    Let $\theta$ be the unique invariant regular Borel probability measure on $X$ \cite[Chapter~7]{auslander1988minimal}. 
    It follows from the invariance of $\theta$ that $\theta(A')>0$ holds for all $A'\in Y$. 
    This allows us to define $\pi\colon Y\to \mathcal{M}(X)$ by $A'\mapsto \theta((\cdot)\cap A')/\theta(A')$. 
    From the invariance of $\theta$ we observe that $\pi$ is equivariant. 
    Furthermore, since each $A'$ is a union of minimal components of $(X,\Gamma_G)$ and since each such component has a positive measure it follows that $\pi$ is injective. 
    Since $Y$ is finite, $\pi$ a conjugation onto its image. 
\end{proof}

\begin{theorem}
\label{the:MXcharacterizationOdometerViaMinimalComponents}
    For an infinite subodometer $(X,G)$ the following statements are equivalent.
    \begin{itemize}
        \item[(i)] $(X,G)$ is an odometer.
        \item[(ii)] All $\mathcal{M}$-quasifactors of $(X,G)$ are factors. 
        \item[(iii)] All finite $\mathcal{M}$-quasifactors of $(X,G)$ are factors. 
    \end{itemize}
\end{theorem}
\begin{proof}
(i)$\Rightarrow$(ii): 
    Whenever $(X,G)$ is an odometer we know from Theorem \ref{the:forOdometersComponentsMXfactor} that all $\mathcal{M}$-quasifactors are factors (up to conjugacy). 

(ii)$\Rightarrow$(iii):
    Immediate. 

(iii)$\Rightarrow$(i):
    From Theorem \ref{the:hyperspaceCharacterizationOdometerViaQuasifactors} we observe that an infinite subodometer that is not an odometer allows for a finite $\hyper$-quasifactor $(Y,G)$ that is not a factor. It follows from Lemma \ref{lem:HXMXtranslation} that $(Y,G)$ is also a $\mathcal{M}$-quasifactor (up to conjugacy). 
\end{proof}

\section{Disjointness of subodometers}
\label{sec:DisjointnessOfSubodometers}
Two minimal actions $(X,G)$ and $(Y,G)$ are called \emph{disjoint} if $(X\times Y,G)$ is minimal. We write $X\perp Y$ to indicate that $X$ and $Y$ are disjoint. 
Note that two disjoint actions only have the trivial action as a common factor. 
The converse is true if at least one of the actions is regular and distal \cite[Theorem~11.21]{auslander1988minimal}; in particular, this holds if one of the actions is an odometer. 
However, it is well known that in general the converse fails. 
See Example \ref{exa:subodometersNoCommonFactors} below for an illustration involving subodometers. 
For further details on disjointness see \cite[Chapter~11]{auslander1988minimal}. 
We next characterize disjointness to subodometers and the absence of nontrivial common factors with subodometers. 
For this we study first the case of finite subodometers. 

\subsection{Finite subodometers}
\label{subsec:DisjointnessOfSubodometers_FiniteSubodometers}

\begin{proposition}
\label{pro:disjointnessNoCommonFactorsFiniteSubodometers}
    Let $\Gamma,\Lambda\in \subfin(G)$. 
    \begin{itemize}
        \item[(i)] 
        $(G/\Gamma,G)$ and $(G/\Lambda,G)$ are disjoint if and only if $\Gamma\Lambda=G$. 
        \item[(ii)] 
        $(G/\Gamma,G)$ and $(G/\Lambda,G)$ have no nontrivial common factor if and only if $\grouphull{\Gamma^g\cup \Lambda}=G$ holds for all $g\in G$. 
    \end{itemize}
\end{proposition}
\begin{remark}
    Note that (i) is a special case of \cite[Theorem~11.12]{auslander1988minimal}. We include the short proof for the convenience of the reader. 
\end{remark}
\begin{proof}
(i):     
    Since $G/\Gamma\times G/\Lambda$ is finite, $(G/\Gamma,G)$ and $(G/\Lambda,G)$ are disjoint if and only if the orbit of $(\Gamma, \Lambda)$  is equal to $G/\Gamma\times G/\Lambda$. Thus, for every $g\in G$, there exists $h\in G$ such that 
    $(h\Gamma, h\Lambda)
    =h.(\Gamma, \Lambda)
    =(\Gamma, g\Lambda)$. 
    From $h\Gamma=\Gamma$ we observe $h\in \Gamma$ and  $g\Lambda=h\Lambda\subseteq \Gamma\Lambda$ yields
    $g\in \Gamma\Lambda$. 

    For the converse assume $G=\Gamma\Lambda$ and consider $(g\Gamma,g'\Lambda)\in G/\Gamma\times G/\Lambda$. 
    There exist $h\in \Gamma$ and $h'\in \Lambda$ with $g^{-1}g'=hh'$, i.e.\ $g'=ghh'$. 
    Thus, we have
    \begin{align*}
        gh.(\Gamma,\Lambda)
        = (gh\Gamma,gh\Lambda)
        = (g\Gamma, ghh'\Lambda)
        = (g\Gamma,g'\Lambda).
    \end{align*}
    This shows that $G/\Gamma\times G/\Lambda$ is minimal.  

(ii): 
    Assume that $(G/\Gamma,G)$ and $(G/\Lambda,G)$ have no nontrivial common factor and consider $g\in G$.
    Since $(G/\Gamma^g)$ is conjugated to $(G/\Gamma,G)$ we observe that $(G/\grouphull{\Gamma^g\cup \Lambda},G)$ is a common factor of the considered actions and hence trivial. 
    This shows $\grouphull{\Gamma^g\cup \Lambda}=G$. 

    For the converse, 
    assume that $\grouphull{\Gamma^g\cup \Lambda}=G$ holds for all $g\in G$ 
    and consider a common factor $(Z,G)$. 
    Choose factor maps $\pi_\Gamma\colon G/\Gamma\to Z$ and $\pi_\Lambda\colon G/\Lambda\to Z$. 
    Since $(Z,G)$ is finite and minimal there exists $g\in G$ with 
    $\pi_\Lambda(\Lambda)=g.\pi_\Gamma(\Gamma)=\pi_\Gamma(g\Gamma)$. 
    We have 
    \begin{align*}
        G_0(\pi_\Lambda(\Lambda))=G_0(\pi_\Gamma(g\Gamma))\supseteq G_0(g\Gamma)=\Gamma^g
    \end{align*}
    and furthermore $G_0(\pi_\Lambda(\Lambda))\supseteq G_0(\Lambda)=\Lambda$. 
    Since $Z$ is finite $G_0(\pi_\Lambda(\Lambda))$ is a group and we observe
    $G_0(\pi_\Lambda(\Lambda))\supseteq \grouphull{\Gamma^g\cup \Lambda}=G$.
    This shows that $\pi(\Lambda)$ is a fixed point of $(Z,G)$. 
    Since $(Z,G)$ is minimal we observe it to be trivial. 
\end{proof}

\begin{remark}
\label{rem:disjointnessFiniteOdometers}
    If one of the groups $\Gamma$ or $\Lambda$ is normal, then $(G/\Gamma, G)$ and $(G/\Lambda,G)$ are disjoint if and only if $G=\grouphull{\Gamma \cup \Lambda}=\Gamma\Lambda$, 
    which is in turn equivalent to the fact that $(G/\Gamma,G)$ and $(G/\Lambda,G)$ 
    have no nontrivial common factors. 
    This also follows from \cite[Theorem~11.21]{auslander1988minimal}, 
    since odometers are regular equicontinuous systems. 
\end{remark} 

The following example illustrates that the testing for all $g\in G$ in Proposition \ref{pro:disjointnessNoCommonFactorsFiniteSubodometers}(ii) is necessary. 

\begin{example}
\label{exa:noCommonFactorsCondition}
    Consider the symmetric group $S_3$, $\Gamma:=\grouphull{(12)}$ and $\Lambda:=\grouphull{(13)}$. 
    For $g:=(23)$ we observe that $\Gamma^g=\Lambda$ and hence that $(G/\Gamma,G)$ and $(G/\Lambda,G)$ are conjugated. 
    Nevertheless, from $(13)(12)=(123)$ we observe 
    \[\grouphull{\Gamma\cup \Lambda}=\grouphull{(12),(13)}=\grouphull{(12),(123)}=S_3.\] 
\end{example}

The following example illustrates that in general the absence of a common factor can occur for non-disjoint finite subodometers. 

\begin{example}
\label{exa:subodometersNoCommonFactors}
    Denote $r:=(1234)$ and $s:=(12)(34)$ and consider $D_4:=\grouphull{r,s}$. 
    Note that $D_4$ is the symmetry group of the square and $|D_4|=8$. 
    Let $\Gamma:=\grouphull{s}$ and 
    $\Lambda:=\grouphull{rs}=\grouphull{(13)}$. 
    From $|\Gamma|=|\Lambda|=2$ we observe that $\Gamma\Lambda\subsetneq D_4$.
    Thus, $(D_4/\Gamma,D_4)$ and $(D_4/\Lambda,D_4)$ are not disjoint.

    Note that we have 
    $s^2=()$, 
    $r^4=()$,
    and $rs=sr^{-1}$. 
    Thus, any $g\in D_4$ can be represented as $g=s^lr^k$ with $l\in \{0,1\}$ and $k\in \{0,1,2,3\}$. 
    From $rs=sr^{-1}$ we observe $r^ks=sr^{-k}$ and hence
    \[\grouphull{\Lambda^g\cup \Gamma}
    \supseteq \Gamma \Lambda^g \Gamma\ni (s^{-l})(g(rs)g^{-1})(s^{l+1})
    =r^krsr^{-k}s=r^{2k+1}.\]
    From $r^4=()$ we know $r^{2k+1}\in \{r,r^{-1}\}$ and observe $r\in \grouphull{\Lambda^g\cup \Gamma}$.
    Thus, $D_4=\grouphull{r,s}\subseteq \grouphull{\Lambda^g\cup \Gamma}$ and it follows that $(D_4/\Gamma,D_4)$ and $(D_4/\Lambda,D_4)$ admit no nontrivial common factor. 
\end{example}

\begin{remark}
    Similar examples with infinite acting groups can easily be constructed from the previous examples by considering an infinite group $H$ and $G:=H\times S_3$ (or $G:=H\times D_4$). 
    With $\Gamma$ and $\Lambda$ as in the respective example 
    $\Gamma':=H\times \Gamma$ and $\Lambda':=H\times \Lambda$ exhibit the same phenomenon. 
\end{remark}

\subsection{Disjointness to subodometers}
\label{subsec:DisjointnessOfSubodometers_DisjointnessToSubodometers}
We next characterize the disjointness of a minimal action to a subodometer in terms of the respective factors. 

\begin{theorem}
\label{the:disjointnessCharacterization}
    Let $(X,G)$ be a minimal action and $(Y,G)$ be a subodometer. 
    The following statements are equivalent:  
    \begin{itemize}
        \item[(i)] $X\perp Y$.
        \item[(ii)] $Z\perp Y$, for every finite factor $Z$ of $X$. 
        \item[(iii)] $X \perp Z'$, for every finite factor $Z'$ of $Y$.
        \item[(iv)] $Z \perp Z'$, for all finite factors $Z$ and $Z'$ of $X$ and $Y$, respectively.
        \item[(v)] $\Lambda\Gamma=G$, for all $\Lambda\in \Eig(X,G)$ and $\Gamma\in \Eig(Y,G)$.
    \end{itemize}
\end{theorem}

\begin{remark}\cite[Proposition~7.1]{auslander1988minimal}
\label{rem:disjointnessFactorInheritance}
    Let $(X,G)$ and $(Y,G)$ be minimal actions.
    If $(Z,G)$ is a factor of $(Y,G)$, then $(X\times Z,G)$ is a factor of $(X\times Y,G)$. Since factors of minimal actions are minimal we observe that $X\perp Y$ implies $X\perp Z$.
\end{remark}

For the proof of Theorem \ref{the:disjointnessCharacterization} we first show the following. 

\begin{lemma}
\label{lem:preparationForDisjointnessCharacterization}
    Let $(X,G)$ be a minimal action, $x\in X$ and $\Gamma\in \subfin(G)$ normal. 
    $B:=\overline{\Gamma.x}$ is clopen and satisfies $\Gamma\subseteq G_0(B)\in \Eig(X,G)$. 
\end{lemma}
\begin{proof} 
    Clearly, $G_0(B)$ is a group and $\Gamma\subseteq G_0(B)$.
    In particular, $G_0(B)$ is a finite index subgroup. 
    Since $(X,G)$ is minimal it is pointwise almost periodic.
    Thus, by Lemma \ref{lem:normalSubgroupMinimalComponents} and the normality of $\Gamma$ the action $(X,\Gamma)$ decomposes into minimal components. 
    Since $\Gamma$ is of finite index there are only finitely many components and hence all components are clopen. 
    Note that one of these components is given by $B$. 
    Denote $\rho$ for the respective equivalence relation and note that $\rho$ is invariant and closed. 
    Considering the factor map $\pi\colon X\to X/\rho$ and $x\in B$ we observe that $G_0(B)=G_0(\rho[x])\in \Eig(X,G)$.
\end{proof}

\begin{proof}[Proof of Theorem \ref{the:disjointnessCharacterization}:]
    From Remark \ref{rem:disjointnessFactorInheritance} we observe that 
    (i)$\Rightarrow$(ii)$\Rightarrow$(iv) 
    and that 
    (i)$\Rightarrow$(iii)$\Rightarrow$(iv).
    From Proposition \ref{pro:disjointnessNoCommonFactorsFiniteSubodometers} we know that (iv)$\Rightarrow$(v). 

'(v)$\Rightarrow$(iii)':
    For any finite factor $(Z',G)$ of $(Y,G)$ there exists $\Gamma\in \Eig(Y,G)$ such that $(G/\Gamma,G)$ and $(Z',G)$ are conjugated. 
    Let $A$ be a minimal subaction of $X\times G/\Gamma$. 
    The projection $A\to G/\Gamma$ is surjective and hence there exists $x_0\in X$ with $(x_0,\Gamma)\in A$. 
    Denote $B:=\overline{\Gamma_G.x_0}$ and $\Lambda:=G_0(B)$. 
    From Lemma \ref{lem:preparationForDisjointnessCharacterization} we know that $B$ is open and that $\Lambda\in \Eig(X,G)$. 
    For $g\in \Gamma_G$ we have $g\Gamma=\Gamma$ and hence $(g.x_0,\Gamma)=g.(x_0,\Gamma)\in g.A$. 
    Since $A$ is invariant and closed we observe $B\times \{\Gamma\}\subseteq A$. 

    To prove that $(X,G)$ and $(G/\Gamma,G)$ are disjoint it remains to show that $A=X\times G/\Gamma$. 
    Consider $(x,y)\in X\times G/\Gamma$. Since $B$ is open the minimality of $(X,G)$ yields the existence of $g_1\in G$ with $x\in g_1.B$. 
    Furthermore, there exists $g_2\in G$ with $y=g_2\Gamma$. 
    Recall that $\Lambda\in \Eig(X,G)$.
    It follows from (v) and Proposition \ref{pro:disjointnessNoCommonFactorsFiniteSubodometers} that $(G/\Lambda,G)$ and $(G/\Gamma,G)$ are disjoint, i.e.\ that $(G/\Lambda\times G/\Gamma,G)$ is minimal. 
    Thus, there exists $h\in G$ with $(h\Lambda,h\Gamma)=(g_1\Lambda,g_2\Gamma)$.
    From $h\Lambda=g_1\Lambda$ we observe $g_1^{-1}h\in \Lambda=G_0(B)$ and hence $x\in g_1.B=h.B$. 
    Thus, $y=g_2\Gamma=h\Gamma$ and the invariance of $A$ yield
    $(x,y)\in h.(B\times \{\Gamma\})\subseteq A$. 

(iii)$\Rightarrow$(i): 
    Recall that $(Y,G)=\varprojlim_{\Gamma\in S} G/\Gamma$ for some scale $S\subseteq \subfin(G)$. 
    It is straightforward to observe that $(X\times Y,G)$ and $\varprojlim_{\Gamma\in S} (X\times G/\Gamma, G)$ are conjugated. 
    By (iii), $X\perp G/\Gamma$ holds for all $\Gamma\in S$.
    Thus, $(X\times Y,G)$ is minimal as the inverse limit of minimal actions. This shows $X\perp Y$. 
\end{proof}

\subsection{Common factors with subodometers}
\label{subsec:DisjointnessOfSubodometers_CommonFactorsWithSubodometers}

\begin{theorem}
\label{the:characterizationNNCFsubodometers}
    Let $(X,G)$ be a minimal action and $(Y,G)$ be a subodometer. 
    The following statements are equivalent. 
    \begin{itemize}
        \item[(i)]$X$ and $Y$ have no nontrivial common factor.
        \item[(ii)]
        $X$ and $Y$ have no nontrivial common finite factor.
        \item[(iii)]
        $\grouphull{\Lambda\cup\Gamma}=G$, for every $\Lambda\in \Eig(X,G)$ and every $\Gamma\in \Eig(Y,G)$.
    \end{itemize}    
\end{theorem}
\begin{remark}
    Since finite actions can be represented in the form $G/\Gamma$ for a suitable choice of $\Gamma$ (ii) can be reformulated as $\Eig(X,G)\cap \Eig(Y,G)=\{G\}$. 
\end{remark}
\begin{proof}
'(i)$\Rightarrow$(ii)': 
    Immediate. 

'(ii)$\Rightarrow$(iii)':
    For $\Lambda\in \Eig(X,G)$ and $\Gamma\in \Eig(Y,G)$ any common factor of the actions 
    $(G/\Lambda,G)$ and $(G/\Gamma,G)$ is a common finite factor of $(X,G)$ and $(Y,G)$ and trivial by (ii). 
    We observe $\grouphull{\Lambda\cup \Gamma}=G$ from Proposition \ref{pro:disjointnessNoCommonFactorsFiniteSubodometers}. 

'(iii)$\Rightarrow$(i)':
    Consider a common factor $(Z,G)$ of $(X,G)$ and $(Y,G)$
    and note that $(Z,G)$ is a subodometer as a factor of the subodometer $(Y,G)$.   
    Thus, Proposition \ref{pro:EIGfactorsAndConjugacyInvariant} yields $\Eig(Z,G)\subseteq \Eig(X,G)\cap \Eig(Y,G)$. 
    For $\Gamma\in \Eig(Z,G)$ we observe from (iii) that $\Gamma=\grouphull{\Gamma\cup \Gamma}=G$. 
    This shows $\Eig(Z,G)=\{G\}$. 
    Since $\Eig$ is a complete conjugacy invariant for subodometers (Proposition \ref{pro:EIGfactorsAndConjugacyInvariant}) we observe that $(Z,G)$ and $(G/G,G)$ are conjugated, i.e.\ that $(Z,G)$ is trivial. 
\end{proof}

\subsection{Common factors and disjointness via scales}
\label{subsec:DisjointnessOfSubodometers_CommonFactorsAndDisjointnessViaScales}

\begin{corollary}
\label{cor:disjointnessAndScales}
    If $(X,G)$ is a minimal action and $(Y,G)$ a subodometer generated by a scale $S$. 
    \begin{itemize}
        \item[(i)]
        $X\perp Y$ if and only if $\Lambda\Gamma=G$ holds for all $\Lambda\in \Eig(X,G)$ and $\Gamma\in S$.
        \item[(ii)] 
        $X$ and $Y$ have no nontrivial common factor if and only if 
        $\grouphull{\Lambda\cup\Gamma}=G$ holds for all $\Lambda\in \Eig(X,G)$ and $\Gamma\in S$.
    \end{itemize}
    If also $(X,G)$ is a subodometer generated by a scale $S'$ the following statements hold. 
    \begin{itemize}
        \item[(iii)] $X\perp Y$ if and only if $\Lambda\Gamma=G$ holds for all $\Lambda\in S'$ and $\Gamma\in S$. 
        \item[(iv)]
        $X$ and $Y$ have no nontrivial common factor if and only if 
        $\grouphull{\Lambda^g\cup\Gamma}=G$ holds for all $\Lambda\in S'$, $\Gamma\in S$ and $g\in G$.
    \end{itemize}
\end{corollary}

\begin{remark}
    Recall from Example \ref{exa:noCommonFactorsCondition} that the testing for all $g\in G$ in (iv) is necessary. 
\end{remark}

\begin{proof}
(i): 
    If $X\perp Y$ we observe the second condition from $S\subseteq \Eig(X,G)$ and Theorem \ref{the:disjointnessCharacterization}. 
    For the converse we show Theorem \ref{the:disjointnessCharacterization}(v) and consider $\Lambda\in \Eig(X,G)$ and $\Gamma\in \Eig(Y,G)$.
    Since $S$ is a scale that generates $(Y,G)$ there exist $g\in G$ and $\Gamma'\in S$ with $\Gamma'\subseteq \Gamma^g$. 
    From our assumption we have $G=\Lambda\Gamma'\subseteq \Lambda\Gamma^g$ and Proposition \ref{pro:disjointnessNoCommonFactorsFiniteSubodometers} yields $G/\Lambda\perp G/\Gamma^g$. Since $(G/\Gamma^g,G)$ and $(G/\Gamma,G)$ are conjugated we observe 
    $G/\Lambda \perp G/\Gamma$, i.e.\ $\Lambda\Gamma=G$.  

(ii):
    From Theorem \ref{the:characterizationNNCFsubodometers} we know that the first condition implies the second. 
    For the converse we show Theorem \ref{the:characterizationNNCFsubodometers}(iii) and consider $\Lambda\in \Eig(X,G)$ and $\Gamma\in \Eig(Y,G)$.
    Since $S$ is a scale that generates $(Y,G)$ there exist $g\in G$ and $\Gamma'\in S$ with $\Gamma'\subseteq \Gamma^g$. 
    From the conjugation invariance of $\Eig(X,G)$ we know $\Lambda^g\in \Eig(X,G)$ and hence
    $G=\grouphull{\Lambda^g\cup \Gamma'}\subseteq \grouphull{\Lambda^g\cup \Gamma^g}=\grouphull{\Lambda\cup \Gamma}^g$. 
    This shows $\grouphull{\Lambda\cup \Gamma}=G$. 

The statements of (iii) and (iv) follow from similar arguments. 
\end{proof}

\begin{remark}
    Recall from Proposition \ref{pro:odometerCharacterizationEIG} that odometers allow for a scale consisting of normal subgroups. We thus rediscover from Remark \ref{rem:disjointnessFiniteOdometers} and Corollary \ref{cor:disjointnessAndScales} that a minimal action is disjoint from an odometer $(Y,G)$ if and only if it shares no nontrivial common factor with $(Y,G)$. 
    As already mentioned, this can also be observed from the regularity and the distality of odometers \cite[Theorem~11.21]{auslander1988minimal}.
\end{remark}

\subsection{Disjointness to all subodometers}
\label{subsec:DisjointnessOfSubodometers_DisjointnessToAllSubodometers}
In \cite[Chapter~11]{auslander1988minimal} it was shown that a minimal action (that admits an invariant regular Borel probability measure) is weakly mixing\footnote{
    An action $(X,G)$ is called \emph{weakly mixing} if every invariant nonempty open subset $U\subseteq X^2$ is dense in $X^2$ \cite{auslander1988minimal}. 
}
if and only if it is disjoint from any equicontinuous action. 
We next classify the minimal actions that are disjoint from all subodometers. 

\begin{theorem}
\label{the:characterizationDisjointnessToSubodometers}
    For a minimal action the following statements are equivalent. 
    \begin{itemize}
        \item[(i)] $X\perp Y$ for all subodometers $(Y,G)$. 
        \item[(ii)] $X\perp Y$ for all odometers $(Y,G)$.
        \item[(iii)] $X\perp Y$ for all finite minimal actions $(Y,G)$. 
        \item[(iv)] $X\perp Y$ for all finite odometers $(Y,G)$.
        \item[(v)] $(X,G)$ has no nontrivial finite factors. 
        \item[(vi)] The maximal equicontinuous factor of $(X,G)$ is connected. 
    \end{itemize}
\end{theorem}
\begin{remark}
    The subodometer $(S_3/\grouphull{(12)},S_3)$ has no nontrivial odometer factor. Thus, considering only odometers in (v) does not yield another characterization of the class of minimal actions disjoint to all subodometers. 
\end{remark}
\begin{proof}[Proof of Theorem \ref{the:characterizationDisjointnessToSubodometers}:]
    Clearly, we have 
    (i)$\Rightarrow$(ii)$\Rightarrow$(iv) and 
    (i)$\Rightarrow$(iii)$\Rightarrow$(iv).

(iv)$\Rightarrow$(i): 
    Denote $S$ for the set of all normal $\Gamma\in \subfin(G)$ and note that $S$ is a scale for the universal odometer $(Y,G)$. 
    Since $(Y,G)$ has all subodometers as factors it suffices to show that $X\perp Y$.    
    From (iv) we observe that $X\perp G/\Gamma$ for all $\Gamma\in S$. In particular, we have $G/\Lambda\perp G/\Gamma$, i.e.\ $\Lambda\Gamma=G$ for all $\Lambda\in \Eig(X,G)$. 
    Thus, Corollary \ref{cor:disjointnessAndScales} yields $X\perp Y$. 

(i)$\Rightarrow$(v):
    Any finite factor $(Y,G)$ of $(X,G)$ is a subodometer and hence satisfies $X\perp Y$. 
    In particular, $(Y,G)$ and $(X,G)$ have no nontrivial common factors and hence $(Y,G)$ must be trivial. 

(v)$\Rightarrow$(vi): 
    Assume that the maximal equicontinuous factor $(Y,G)$ of $(X,G)$ is not connected. 
    There exists a nontrivial clopen subset $A\subseteq Y$. 
    Denote $B:=Y\setminus A$ and note that $\rho:=A^2\cup B^2$ is a clopen equivalence relation. 
    In particular, we have $\rho\in \mathbb{U}_X$. 
    Since $(Y,G)$ is equicontinuous there exists an invariant $\delta\in \mathbb{U}_X$ with $\delta\subseteq \rho$. 
    Clearly, $R:=\bigcap_{g\in G} g.\rho$ is an invariant closed equivalence relation and satisfies $\delta\subseteq R$. 
    We thus have $R\in \mathbb{U}_X$. 
    For $y\in Y$ we observe that $B_R(y)=R[y]$ is a neighbourhood of $y$. 
    Since $R$ is an equivalence relation all equivalence classes are thus open. 
    It follows that there only exist finitely many equivalence classes.
    In particular, the quotient $(Y/R,G)$ establishes a nontrivial finite factor of $(Y,G)$. Since $(Y,G)$ is a factor of $(X,G)$ this contradicts (v). 

(vi)$\Rightarrow$(i):
    Let $(Y,G)$ be a subodometer and consider a finite factor $(Z,G)$ of $(X,G)$. 
    $(Z,G)$ is equicontinuous and hence a factor of the maximal equicontinuous factor of $(X,G)$. 
    Since continuous images of connected sets are connected (vi) yields that $(Z,G)$ is connected and hence trivial. 
    In particular, it satisfies $Z\perp Y$. 
    This shows $Z\perp Y$ for any finite factor $(Z,G)$ of $(X,G)$ and Theorem \ref{the:disjointnessCharacterization} yields $X\perp Y$. 
\end{proof}

\begin{example}
    Whenever $(X,G)$ is a minimal action on a connected compact Hausdorff space, then the maximal equicontinuous factor is connected. 
    Prominent examples are given by irrational rotations on the circle, 
    the distal skew product shift \cite[Chapter~7]{auslander1988minimal}, and
    the actions provided in the field of surface dynamics. 
\end{example}

The following example demonstrates that minimal actions on Stone spaces can be disjoint from all subodometers. 

\begin{example}
    The Sturmian subshift is a minimal action of $\mathbb{Z}$ on the Cantor set. It is well known that its maximal equicontinuous factor is a circle rotation and hence connected. 
\end{example}

\bibliographystyle{alpha}
\bibliography{ref}
\end{document}